\documentclass[12pt]{amsart}
\usepackage{amsmath,amssymb,amscd,mathrsfs,amsthm,amscd,inputenc,enumerate,amsfonts,pdfsync}

\usepackage{color}
\usepackage{graphicx}
\usepackage{mathtools}
\usepackage{verbatim}

\newcommand{\spec}{{\mbox{\rm\texttt{Spec}}}}
\newcommand{\Zar}{{\mbox{\rm\texttt{Zar}}}}
\newcommand{\zar}{{\mbox{\rm\texttt{Zar}}}}
\newcommand{\Kr}{{\mbox{\rm\texttt{Kr}}}}
\newcommand{\kr}{{\mbox{\rm\texttt{Kr}}}}

\newcommand{\qspec}{{\mbox{\rm\texttt{QSpec}}}}

\newcommand{\ms}{\mathscr}

\newcommand{\z}{{\ldots}}

\newcommand{\inssubmod}{\boldsymbol{\overline{F}}}
\newcommand{\smod}{\mathtt{SMod}}

\newcommand{\insfineab}{\mbox{\texttt{SStar}$_{f,\texttt{eab}}$}}

\newcommand{\xcal}{{\boldsymbol{\mathcal{X}}}}

\newcommand{\ucal}{{\boldsymbol{\mathcal{U}}}}

    \newcommand{\FF}{\boldsymbol{\overline{F}}}

    \DeclareMathOperator{\chius}{\mbox{\texttt{Cl}}}
    \DeclareMathOperator{\chiusinv}{\mbox{\texttt{Cl}}^{\mbox{\tiny\texttt{inv}}}}
       \DeclareMathOperator{\chiuscons}{\mbox{\texttt{Cl}}^{\mbox{\tiny\texttt{cons}}}}
       \DeclareMathOperator{\Rd}{\mbox{\texttt{Rd}}}

    \newcommand{\inssemistar}{{\mbox{\rm\texttt{SStar}}}}
     \newcommand{\inssemistabft}{\widetilde{{\mbox{\rm\texttt{SStar}}}}}
          \newcommand{\inssemistarft}{{\mbox{\rm\texttt{SStar}}}_{\!_f}}
 
  \newcommand{\s}{{\mbox{\rm\texttt{s}}}}

 \DeclareMathOperator{\Over}{{\mbox{\rm\texttt{Overr}}}}
   \newcommand{\overr}{{\mbox{\rm\texttt{Overr}}}}
\newcommand{\overric}{{\mbox{\rm\texttt{Overr}}}_{\!\mbox{\tiny\it\texttt{ic}}}}

\newtheoremstyle{mio}%
	{}{} 
	{\itshape}{} 
	{\bfseries}{.}{ } 
	{#1 #2\thmnote{\mdseries~(\scshape #3)}} 
\theoremstyle{mio}
\newtheorem{teor}{Theorem}[section]
\newtheorem{cor}[teor]{Corollary}
\newtheorem{prop}[teor]{Proposition}
\newtheorem{lemma}[teor]{Lemma}

\theoremstyle{definition}

\newtheorem{ex}[teor]{ \textbf{Example}}
\newtheorem{oss}[teor]{Remark}

\title[Inverse-closed subsets of a spectral space]{The upper Vietoris topology on the space of inverse-closed subsets of a spectral space and applications}

\date{\bf\today}

\author{Carmelo A. Finocchiaro}
\address{\hskip -11pt Institute of Analysis and Number Theory, University of Technology, Graz,
	 Steyrergasse 30/II,  8010 Graz, Austria \newline
	 {\; \texttt{present address:}} Dipartimento di Matematica, Universit\`a degli Studi di Padova, Via Trieste 63, 35121 Padova, Italy}
\email{finocchiaro@math.tugraz.at \ or \ carmelo@math.unipd.it}

\author{Marco Fontana}
\address{\hskip -11pt Dipartimento di Matematica e Fisica, Universit\`a degli Studi
	``Roma Tre'',   Largo San Leonardo Murialdo, 1, 00146 Roma, Italy}
\email{fontana@mat.uniroma3.it}

\author{Dario Spirito}
\address{\hskip -11pt Dipartimento di Matematica e Fisica, Universit\`a degli Studi
	``Roma Tre'',  Largo San Leonardo Murialdo, 1, 00146 Roma, Italy}
\email{spirito@mat.uniroma3.it}



\thanks{This work was partially supported by {\sl GNSAGA} of {\sl Istituto Nazionale di Alta Matematica}. The first named author was also supported by a Post Doc Grant from the University of Technology of Graz (Austrian Science Fund (FWF): P 27816).}

\begin{document}


\begin{abstract} 
Given an arbitrary spectral space $X$, we consider the set $\xcal(X)$ of all  nonempty  subsets of $X$ that are closed with respect to the inverse topology. We introduce a Zariski-like topology on $\xcal(X)$ and,  after observing that it coincides the upper Vietoris topology, we prove that $\xcal(X)$ is itself a spectral space, that this construction is functorial, and that $\xcal(X)$ provides an extension of $X$ in a more ``complete'' spectral space. Among the applications, we show  that, starting from an integral domain $D$, $\xcal(\spec(D))$ is homeomorphic to the (spectral) space of all the stable semistar operations of finite type on $D$.
\end{abstract} 

\maketitle

\section{Introduction}
 The first study  of the set of prime ideals from a topological point of view is due to  M. H. Stone \cite{stone1,stone2}, who developed the theory in the context of distributive lattices and Boolean algebras. Later, M. Hochster \cite{ho} defined a \emph{spectral space} as a topological space that  it is homeomorphic to the prime spectrum of a (commutative) ring  endowed with the Zariski topology, and proceeded to show that this class of topological spaces can be characterized in a purely topological way. More precisely, he proved that a topological space $X$ is spectral if and only if it is T$_0$, quasi-compact, it admits a basis of  quasi-compact open subspaces that is closed under finite intersections, and it is sober (i.e., every irreducible closed  subset of $X$ has a (unique) generic point). Spectral spaces can also be viewed through the lens of ordered topological spaces (via the concept of \emph{Priestley space}) \cite{cornish,prie1,prie2}, of bitopological spaces (through \emph{pairwise Stone spaces}) \cite{bezha}, or through domain theory (using the notion of \emph{stably compact space}) \cite{lawson}.

The first example of a spectral space which occurs naturally in commutative algebra, but it is not defined as a spectrum, is the Riemann-Zariski space $\zar(K|D)$ of all the valuation domains  with quotient field $K$ and containing $D$; this was proved by providing explicitly a B\'ezout domain whose prime spectrum is naturally homeomorphic to $\zar(K|D)$ (see \cite{do-fo}, \cite{fifolo2}  and \cite{hk}). Recently,  several other spaces, naturally occurring in multiplicative ideal theory, have been shown to be spectral: for example, this happens for the spaces $\Over(D)$ and $\overric(D)$  consisting, respectively, of  the overrings and of  the integrally closed overrings of an integral domain $D$. This result was later extended to the space $\inssemistarft(D)$ of all semistar operations of finite type on $D$, providing an appropriate and natural topological extension of the spectral space $\Over(D)$ (and, in particular, of both $\spec(D)$ and $\zar(K|D)$) \cite{FiSp}. Unlike the proof of the spectrality of $\zar(K|D)$, these spaces were shown to be spectral using a criterion based on ultrafilters \cite{Fi}, which is well-suited to this kind of spaces; however, this criterion is not constructive,  that is, it does not provide explicitly a ring whose spectrum is homeomorphic to the given  spectral space.

If $X$ is a topological space,   we denote by  $X^{\mbox{\tiny{{\texttt{d}}}}}$   the set $X$ endowed with the \emph{co-compact topology}, i.e., the topology on $X$ having, as a base of open sets, the complements of the subsets of $X$ that are both quasi-compact and obtained as an intersection of open sets \cite[Definition O-5.10]{GHKLMS}. In the context of spectral spaces, the co-compact topology of $X$ is called the \emph{inverse topology} of $X$, and plays a crucial r\^ole in Hochster's study of spectral spaces; it owes its name to the fact that the order canonically associated to the inverse topology coincides with the reverse order of that induced by the spectral topology. Subsets of a spectral space that are closed in the inverse topology are strictly related to the study of representations of integrally closed domains as intersections of collections of valuation domains (see also \cite{olb2010}, \cite{olb2011}, and \cite{olb2015}),  and they represent a way to classify several distinguished classes of semistar operations  of finite type: it was shown in \cite{fifolo2} and \cite{FiSp} that complete (or, \texttt{e.a.b.})  semistar operations  (respectively, stable semistar operations - definitions recalled later)  correspond to the subsets of $\Zar(D)$ (respectively, $\spec(D)$) that are closed in the inverse topology. Moreover, these two spaces are spectral extensions of the spaces $\Zar(D)$ and $\spec(D)$ (see also \cite{fi-fo-sp-survey}).

The aim of this paper is to study, for an arbitrary spectral space $X$, the space $\xcal(X)$ of all nonempty  subsets of $X$ that are closed with respect to the inverse topology; in particular, this study is carried out using the same ultrafilter-theoretic approach of \cite{Fi} and \cite{fifolo2}, and using techniques closer to commutative algebra than to general topology, in an attempt   to bridge the gap between the two communities. After endowing $\xcal(X)$ with a natural topology, we show that it is a spectral space and a spectral extension of the original space $X$.   It is worth no\-ting that this construction,   arisen in the topological context associated to commutative ring theory, is a special case of the construction of the {\sl Smyth powerdomain} of a general topological space $X$, endowed with the upper Vietoris topology  (\cite{V-1922} and \cite{michael}; the definitions are recalled later),  which is usually studied from the point of view of domain theory (see \cite[Section 5]{lawson} and \cite{smyth}).  In Section \ref{sect:appl} we see  that the two spaces of distinguished semistar operations recalled above are examples of the space $\xcal(X)$, when applied to the spectral spaces $X=\Zar(D)$ and $X=\spec(D)$. We  also show that the extension $X\hookrightarrow\xcal(X)$ represents, in a certain sense, a  spectral ``completion'' of the original space $X$, matching the possibility of extending the spectral space $\Over(D)$ inside  the  more ``complete'' spectral space of the semistar operations of finite type $\inssemistarft(D)$. The ``completeness'' mentioned above is related to  a universal-like property satisfied by $\xcal( X)$: broadly speaking, $\xcal( X)$ is the completion of $X$ with respect to the existence of the supremum for families of quasi-compact subspaces.   

\smallskip
We thank the referee for his/her thorough reports and highly appreciate the constructive comments and suggestions on the connections with recent results in domain theory, which significantly contributed to improving the quality of the paper and gave us the opportunity to connect two  far apart strands of research.


\section{Preliminaries}


It is well known that the prime spectrum of a commutative ring endowed with the Zariski topology  is always $T_0$ and quasi-compact, but almost never Hausdorff (it is Hausdorff only in the zero-dimensional case).
Thus, many authors have considered a finer topology on the prime spectrum of a ring, known as the \emph{constructible topology} (see \cite{ch}, \cite [pages 337-339]{EGA} or \cite[Chapter 3,  Exercises 27, 28 and 30]{am}) or as the \emph{patch topology} \cite{ho}.

 Following \cite{prie2} or \cite{sch-tr}, it is possible to introduce  the \emph{constructible topology}  by a Kuratowski closure ope\-rator: if    $ X$ is a spectral space, for each subset $Y$  of $X$, we set:
$$
\begin{array}{rl}
\chiuscons(Y) \! :=  \!\bigcap \{U \!\cup\! (X\!\setminus\! V) \mid  & \hskip -5pt \mbox{ $U$ and $V$ open and quasi-compact in }  X,  \\
 & \hskip -4pt U  \!\cup \! (X \!\setminus \! V) \supseteq Y \}\,.
 \end{array}
$$
We denote by $X^{\mbox{\tiny\texttt{cons}}} $ the set $X$, equipped with the constructible topology.
For Noetherian topological spaces, this definition of constructible topology coincides with the classical one given in \cite{ch}. 
It is well known that the constructible topology  is a refinement of the given topology and it  is always 
Hausdorff.   

\smallskip

 Given a topology on a set $X$, we can define an order $\leq_X$ on $X$ by setting $x\leq_X y$ if $y\in \chius(\{x\})$, where $\chius(Y)$ denotes the closure of a subset $Y$ of $X$. This order is the opposite of the specialization order generally used in topology; however, it is the one more commonly used in commutative algebra and algebraic geometry, since on the spectrum of a ring it coincides with the set-theoretic containment (for example, this is the order used in \cite{ho}). The set
$$
Y^{\mbox{\tiny\texttt{gen}}}  :=\downarrow\! Y:= \{x\in X \mid   x\leq y,\mbox{ for some }y\in Y \}
$$
is called \textit{closure under generizations of $Y$}. 
Similarly, using the opposite order,  the set
$$
Y^{\mbox{\tiny\texttt{sp}}} :=\uparrow\! Y:= \{x\in X \mid    y\leq x,\mbox{ for some }y\in Y \}
$$
is called \textit{closure under specia\-li\-zations of $Y$}. We say that $Y$ is \textit{closed under generi\-zations}   or a \textit{down set} (respectively, \textit{closed under specia\-li\-zations}  or an \textit{upper set})  if $Y=Y^{\mbox{\tiny\texttt{gen}}}$ (respectively, $Y=Y^{\mbox{\tiny\texttt{sp}}}$). 
It is straightforward that, for two elements $x, y$ in a spectral space $X$, we have: 
 $$
 x \leq y \quad \Leftrightarrow \quad \{x \}^{\mbox{\tiny\texttt{gen}}} \subseteq \{y \}^{\mbox{\tiny\texttt{gen}}} \quad \Leftrightarrow \quad\{x \}^{\mbox{\tiny\texttt{sp}}} \supseteq \{y \}^{\mbox{\tiny\texttt{sp}}}\,.
 $$

 Given  a spectral space $X$, M. Hochster \cite[Proposition 8]{ho} introduced a new topology on $ X$, that we call here the  {\it inverse topology}, by defining a Kuratowski closure operator, for each subset $Y$  of $X$, as follows:
  $$
 \chius^{\mbox{\tiny\texttt{inv}}}(Y) :=
\bigcap \{U   \mid  \mbox{ $U$ open and quasi-compact in } X,  \; 
U  \supseteq Y \}\,.
$$
If we denote by  $X^{\mbox{\tiny\texttt{inv}}}$ the set $X$ equipped with the inverse topology, Hochster proved that $X^{\mbox{\tiny\texttt{inv}}}$ is still a spectral space and the partial order on $X$ induced by the inverse topology is the opposite order of that induced by the given topology on $X$. 
In particular, the closure under generizations $\{x\}^{\mbox{\tiny\texttt{gen}}}$ of a singleton is closed in the inverse topology  of $X$, since   $\{x \}^{\mbox{\tiny\texttt{gen}}}=\bigcap\{U \mid  U \subseteq X\mbox{ quasi-compact and open},\, x\in U \}$  \cite[Proposition 8]{ho}.  
On the other hand, it is trivial, by the definition,  that the closure under specializations of a singleton $\{x \}^{\mbox{\tiny\texttt{sp}}}$ is closed in the given topology of $X$, since $\{x \}^{\mbox{\tiny\texttt{sp}}}= \chius(\{x\})$.

Recall that it is well known that  $\chiusinv(Y)=(\chiuscons(Y))^{\mbox{\tiny\texttt{gen}}}$  (see for instance \cite[Lemma 1.1]{fontana-AMPA} applied to the inverse topology or, explicitly,
 \cite[Remark 2.2]{fifolo2}; a more general situation is considered in \cite[Section 2.2]{lawson}).
It follows that each closed set in the inverse topology  (called for short, \emph{inverse-closed}) is closed under generizations  and, from \cite[Proposition 2.6]{fifolo2},  that a quasi-compact subspace $Y$ of $X$ closed for generizations is inverse-closed.  

\medskip

 We would like to mention here the existence of  several different point of views that might shed  further light on the theory of spectral spaces.
 
One perspective is through the language of \emph{ordered topological spaces}. Let $X$ be a topological space and $\leq$ an order on $X$: then, the pair $(X,\leq)$ is a \emph{Nachbin space} if $X$ is quasi-compact and the set $\{(x,y)\in X\times X \mid x\leq y \}$ is closed in $X\times X$. A \emph{Priestley space} is a Nachbin space $(X,\leq)$ such that, for every $x,y\in X$ with $x\nleq y$, there exists a clopen subset $\Gamma$ of $X$ that is closed under specialization (with respect to $\leq$) such that $x\in \Gamma$ and $y\notin \Gamma$. It is well known that there is an isomorphism between the category of Priestly spaces (and continuous monotone maps) and the category of spectral spaces (and spectral maps): if $X$ is a spectral space, and $\leq$ is the order induced by the topology, then $(X^{\mbox{\tiny\texttt{cons}}},\leq)$ is a Priestley space, while if $(X,\leq)$ is a Priestley space, then the topology on $X$ generated by the open subsets of $X$ that are closed under generizations (with respect to $\leq$) is a spectral space. In this context, reversing the order defining a Priestley space amounts to passing from a spectral topolgy to its inverse topology, while the case when $\leq$ is the indiscrete order (i.e., $x\leq y$ if and only if $x=y$) corresponds to the case where the spectral space $X$ is Hausdorff, i.e., when the topology on $X$ is equal to its own constructible topology. For a deeper insight on this topic, see, for instance, \cite{bezha,cornish,prie1,prie2} and \cite[Chapter VI]{gierz}.

Another point of view is offered by domain theory. A topological space  $X$ is said to be \emph{stably compact} (see for instance \cite{lawson}) if it satisfies the following properties:
\begin{enumerate}
\item[(a)] $X$ is T$_0$ and quasi-compact.
\item[(b)] $X$ is \textit{locally quasi-compact} (that is, for any open set $U$ of $X$ and any $x\in U$, there are a quasi-compact subspace $K$ of $X$ and an open set $V\subseteq X$ such that $x\in V\subseteq K\subseteq U$). 
\item[(c)] $X$ is \emph{coherent} (that is, any finite intersection of quasi-compact saturated subsets of $X$ is quasi-compact). 
\item[(d)] $X$ is sober. 
\end{enumerate}
\smallskip

  Note that stably compact spaces can also be defined as the retracts of the spectral spaces \cite[Lemma 3.13(b)]{simmons};  further   connections are outlined in the following well known results.

Recall that a subset of a topological space is called a \emph{saturated subset} if it is an intersection of a family of open sets.

\begin{lemma}\label{uno}
Let $X$ be a topological space having a basis   for the open sets given by the quasi-compact open subspaces. 
\begin{enumerate}[\,\,\,\,\,\,\rm (1)]
\item If $K\subseteq U\subseteq X$, $K$ is quasi-compact and $U$ is open   in $X$, then there exists   a quasi-compact open subspace $\Omega$ of $X$ such that $K\subseteq \Omega\subseteq U$. 
\item If $X$ is spectral, then a subset of $X$ is closed, with respect to the inverse topology, if and only if it is saturated and quasi-compact. 
\end{enumerate}
\end{lemma}
%

Under this terminology, a spectral space is exactly a stably compact space such that the quasi-compact open subspaces are a basis:

\begin{lemma}\label{due}
Let $X$ be a topological space. Then, the following conditions are equivalent.
\begin{enumerate}
\item[\rm (i)] $X$ is a spectral space. 
\item[\rm (ii)]  $X$ is a stably compact space with a basis   for the open sets given by the  quasi-compact open subspaces. 
\end{enumerate}
\end{lemma}
%

Note that the notion of stably compact space is strictly more general of that of spectral space. For instance, it is easy to see that the subspace $[0, 1]$ of the real line is stably compact but not a spectral space, for lack of quasi-compact open subspaces.

Finally, we observe that the isomorphism  between the category of Priestley spaces and spectral spaces (recalled above)  extends naturally to an   isomorphism  between the categories of Nachbin spaces (and continuous monotone maps) and of stably compact spaces (and proper maps). See \cite[Chapter VI]{gierz}.


 \medskip

\section{The space of inverse-closed subsets of a spectral space}

Let $X$ be a spectral space. The main object of this paper is the space
\begin{equation*}
\boldsymbol{\mathcal{X}}(X) := \{ Y \subseteq X \mid   Y \neq \emptyset,\,  Y =\chiusinv(Y)\},
\end{equation*}
that is, $\xcal(X)$ is the set of all nonempty subset of $X$ that are closed in the inverse topology. From the point of view of ordered topological spaces, if $(X,\leq)$ is a Priestley space, then $\xcal(X)$ is the space of nonempty closed downsets of $X$.

If $X$ is understood from the context, we shall simply write $\xcal$ instead of $\xcal(X)$. If $X=\spec(R)$ for some ring $R$, we write for short $\xcal(R)$ instead of $\xcal(\spec(R))$.

 We define a \textit{Zariski topology} on $\boldsymbol{\mathcal{X}}(X)$ by taking, as subbasis of open sets, the sets of the form
\begin{equation*}
\boldsymbol{\mathcal{U}}(\Omega):=\{Y\in \boldsymbol{\mathcal{X}} \mid Y\subseteq \Omega \},
\end{equation*}
where $\Omega$ varies among the quasi-compact open subspaces of $X$. 
Note that the previous subbasis is in fact a basis, since $\boldsymbol{\mathcal{U}}(\Omega)\cap \boldsymbol{\mathcal{U}}(\Omega')=\boldsymbol{\mathcal{U}}(\Omega\cap \Omega')$ and $\Omega\cap \Omega'$ is a quasi-compact open subspace of $X$, for any pair $\Omega,\Omega'$ of quasi-compact open subspaces  of $X$.
Moreover, $\emptyset \neq \Omega \in \boldsymbol{\mathcal{U}}(\Omega)$, since a quasi-compact open subset $\Omega$ of $X$ is a closed in the inverse topology of $X$. Note also that,  when $X=\spec({R})$, for some ring $R$, a generic basic open set of the Zariski topology on $\boldsymbol{\mathcal{X}} (R)$ is of the form
 $$
\boldsymbol{\mathcal{U}}(J) := \ \boldsymbol{\mathcal{U}}(\texttt{D}(J))=\{Y\in \boldsymbol{\mathcal{X}}(R) \mid Y\subseteq \texttt{D}(J) \},
$$
where $J$ is any finitely generated ideal of $R$.

The construction $\xcal(X)$ can also be understood in terms of the traditional domain-theoretic definition of the Smyth powerdomain   in  the setting of  topological spaces.
More precisely,
let $X$ be a topological space. Following, for example, \cite[Definition 5.2]{lawson}, the \emph{Smyth powerdomain of $X$} is the collection $\boldsymbol{\mathcal Q}(X)$ of all nonempty quasi-compact saturated subsets of $X$, equipped with the \emph{upper Vietoris topology}, that is, the topology on $\boldsymbol{\mathcal Q}(X)$ whose basic open sets are the sets of the form
$$
U^+:=\{Q\in \boldsymbol{\mathcal Q}(X)\mid Q\subseteq U \},
$$
for any open set $U$ of $X$. 

In view of  Lemma \ref{uno}(1),
if $X$ is a spectral space, then $\boldsymbol{\mathcal Q}(X)=\boldsymbol{\mathcal{X}}(X)$, as sets.
Now,  we show that this equality holds at a topological level. 

\begin{prop}\label{zariski-vietoris}
	Let $X$ be a spectral space. Then, the space $\boldsymbol{\mathcal{X}}(X)$, endowed with the Zariski topology, coincides with the space   ${\boldsymbol{\mathcal Q}}(X)$, endowed with the upper Vietoris topology.
\end{prop}
\begin{proof}
	Clearly, it is sufficient to show that, if $U$ is an open subset of $X$, then $U^+$ is open, with respect to the Zariski topology  on $\boldsymbol{\mathcal{X}}(X)$. Take a set $Q\in U^+$. Since $Q$ is, in particular, quasi-compact,  Lemma  \ref{uno}(1)
	implies the existence of    a quasi-compact  open  subspace $\Omega$ of $X$ such that $Q\subseteq \Omega\subseteq U$. It follows immediately that $U\in \boldsymbol{\mathcal U}(\Omega)=\Omega^+\subseteq U^+$. The proof is now complete. 
\end{proof}

On the other hand, from the theory of stably compact spaces,
the following property hold.

\begin{teor}\cite[Theorem 5.9]{lawson}\label{stably-comp-vietoris}
	Let $X$ be a stably compact space. Then the Smyth powerdomain $\boldsymbol{\mathcal Q}(X)$ of $X$, equipped with the upper Vietoris topology, is stably compact. 
\end{teor}

By the previous Lemma \ref{due},
the fact that $\boldsymbol{\mathcal{X}}(X)$ is a spectral space can be seen in the frame of the theory of stably compact spaces.  We start with the following easy lemma, whose proof is left to the reader.

\begin{lemma}
Let $X$ be any spectral space. Then $\boldsymbol{\mathcal{X}}(X)$, endowed with the Zariski topology, is a T$_0$-space.
\end{lemma}

\begin{teor} \label{embedding} 
Let $X$ be a spectral space.
\begin{itemize}
\item[(1)] The space
 $\boldsymbol{\mathcal{X}}:= \boldsymbol{\mathcal{X}}(X)$, 
 endowed with the Zariski topology  (i.e., with the upper Vietoris topology), is a spectral space.
 \item[(2)] Let $Y_1, \ Y_2 \in \boldsymbol{\mathcal{X}}$. Then, 
$ Y_1 \subseteq Y_2$   if and only if   $ Y_1\leq _{\boldsymbol{\mathcal{X}}} Y_2$.
 \item[(3)] The canonical map $\varphi:X\rightarrow \boldsymbol{\mathcal{X}}$, defined by
$
\varphi(x):=\{x\}^{\mbox{\tiny{\emph{\texttt{gen}}}}}
$, for each $x\in X$,   is  a spectral embedding  (and, in particular, an order-preserving embedding  between ordered sets, with the ordering induced by the Zariski topologies). 
\item[(4)]   $\boldsymbol{\mathcal{X}}$ has a unique 
maximal point (i.e., $X$). 
\end{itemize}
\end{teor}
\begin{proof}
(1) Let $\ucal(\Omega)$ be a member of the canonical basis of $\xcal(X)$, where $\Omega\neq\emptyset$ is a quasi-compact open subspace of $X$. If $\mathcal{A}$ is an open cover of $\ucal(\Omega)$, then there is a set $A\in\mathcal{A}$ such that $\Omega\in A$.
Hence, there is a nonempty quasi-compact open set $V$ of $X$ such that $\Omega\in\ucal(V)\subseteq A$. If now $U\in\ucal(\Omega)$, then $U\subseteq\Omega\subseteq V$, and thus $U\in A$; it follows that the singleton $\{A\}$ is an open subcover of $\ucal(\Omega)$. Therefore, $\ucal(\Omega)$ is quasi-compact.

By Proposition \ref{zariski-vietoris} and  Theorem  \ref{stably-comp-vietoris}, $\xcal(X)$ is stably compact; by Lemma \ref{due}, and the previous reasoning, it follows that $\xcal(X)$ is a spectral space.

Statements (2), (3) and (4) are straightforward.
\end{proof}

\begin{oss}
  As it was done in the first version of the present paper, it is also possible to prove the spectrality of $\xcal(X)$ by using, instead of \cite[Theorem 5.9]{lawson},  ultrafilter-theoretic techniques developed by ring theorists  for studying spectral spaces; we sketch how to do it. 
By \cite[Corollary 3.3]{Fi}, it suffices to show that, if $\ms U$ is an ultrafilter on $\boldsymbol{\mathcal{X}}$, then the set 
\begin{equation*}
	\boldsymbol{\mathcal{X}}_{\boldsymbol{\mathcal{T}}}(\ms U): =\{Y\in \boldsymbol{\mathcal{X}} \mid  [\mbox{for each }\boldsymbol{\mathcal{U}}(\Omega), \;\; 
	Y\in \boldsymbol{\mathcal{U}}(\Omega)\Leftrightarrow \boldsymbol{\mathcal{U}}(\Omega)\in \ms U]\,  \}
\end{equation*}
is nonempty. Set    
\begin{equation*}
\mathscr F(\ms U) := \{\Omega \mid \Omega\subseteq X \mbox{ quasi-compact open and }\boldsymbol{\mathcal{U}}(\Omega)\in \ms U\}\,.
\end{equation*}
Then, $\mathscr F(\ms U)$ does not contain the empty set and has the finite intersection property; therefore,
$$
Y_0:=\bigcap\{\Omega \mid \Omega \in \mathscr F(\ms U)  \}
$$
is a nonempty inverse-closed subset of $X$, i.e., $Y_0\in\xcal(X)$.

Furthermore, if $Y_0\in\ucal(\Omega_0)$ and $\Omega_0\notin\mathscr{U}$, then, since $\mathscr{U}$ is closed by finite intersection,
	$$
	\mathscr C:=\{\Omega\cap (X\setminus\Omega_0) \mid \Omega \mbox{ quasi-compact open in $X$ and } \boldsymbol{\mathcal{U}}(\Omega)\in\ms U    \}
	$$
	is a collection of sets having the finite intersection property, and each element of $\mathscr{C}$ is closed in the constructible topology. Therefore, its intersection is nonempty, and any point in the intersection belongs to $Y_0\setminus\Omega_0$, a contradiction. Thus $\Omega_0\in\mathscr{U}$. Conversely, if $\Omega_0\in\mathscr{U}$ then 
	$$
	\Omega_0\supseteq \bigcap\{\Omega \mid \Omega \subseteq X \mbox{ quasi-compact open and }\boldsymbol{\mathcal{U}}(\Omega)\in\ms U \}=Y_0
	$$
	i.e., $Y_0\in \boldsymbol{\mathcal{U}}(\Omega_0)$. Hence, $Y_0\in\boldsymbol{\mathcal{X}}_{\boldsymbol{\mathcal{T}}}(\ms U)$ and $\xcal(X)$ is a spectral space.
\end{oss}

\begin{oss}\label{oss-inv} {\bf (a)}
Let $X$ be a spectral space and, as above, let $X^{\mbox{\tiny{{\texttt{inv}}}}}$ denote the set $X$, endowed with the inverse topology. 
Then, keeping in mind the Hochster's duality (i.e., sketchy, $(X^{\mbox{\tiny{{\texttt{inv}}}}})^{\mbox{\tiny{{\texttt{inv}}}}} =X$),
 the set $\xcal'(X):=\xcal(X^{\mbox{\tiny{{\texttt{inv}}}}})$ consists of all the nonempty closed sets of $X$, with respect to the given spectral topology. 
 Keeping in mind that the quasi-compact open subspaces of $X^{\mbox{\tiny{{\texttt{inv}}}}}$ are precisely the complements of the quasi-compact open subspaces of $X$, it follows immediately, by definition, that the Zariski topology of $\xcal'(X)$ has as a basis of open sets the collection of the sets of the type:
$$ \ucal'(\Omega):=\ucal(X\!\setminus\!\Omega)=\{C\in\xcal'(X)\mid C\cap \Omega=\emptyset  \}\,,
$$
for $\Omega$ varying among the quasi-compact open subspaces of $X^{\mbox{\tiny{{\texttt{inv}}}}}$.
Dually, the cano\-ni\-cal map $\varphi':X^{\mbox{\tiny{{\texttt{inv}}}}}\rightarrow \xcal'(X)$, defined by
 $x\mapsto \{x\}^{\mbox{\tiny{{\texttt{sp}}}}}$, is a spectral topolo\-gical embedding.
Now, let $X$ be the prime spectrum of a ring $R$, and let
$\Rd(R)$ be the set of all proper radical ideals of $R$, endowed with the so called {\it hull-kernel topology}, that is the topology whose subbasic open sets are those of the form $\texttt{D}(x_1, x_2, \z,x_n):=\{H \in \Rd(R)\mid (x_1,x_2, \z,x_n)R\nsubseteq H \}$. 
 In \cite{fi-fo-sp-MANUSCRIPTA},  it is proved that $\Rd(R)$ is a spectral space that extends naturally the space $\spec(R)$, endowed with the Zariski topology.
  Moreover, it is proved that  there is a canonical homeomorphism 
  $\lambda: \xcal'(R):= \xcal'(\spec(R))\rightarrow \Rd(R)^{\mbox{\tiny{{\texttt{inv}}}}}$, 
  mapping a nonempty closed set $C\subseteq \spec(R)$ to the radical ideal $\lambda(C):=\bigcap\{P \mid P\in C \}$.
             
  {\bf (b)}  
Recall that, for any topological space $X$, the {\it co-compact topology} on $X$ is the topology having as a base for the open sets the complements of quasi-compact saturated  
  subsets of $X$ \cite[Definition O-5.10]{GHKLMS}. The topological space $X$ endowed with this topology, denoted by $X^{\mbox{\tiny{{\texttt{d}}}}}$, is called the {\it de Groot dual} of $X$. It is known that, if $X$ is a stably compact space, $X^{\mbox{\tiny{{\texttt{d}}}}}$ is also stably compact and $(X^{\mbox{\tiny{{\texttt{d}}}}})^{\mbox{\tiny{{\texttt{d}}}}} = X$ \cite[Proposition 3.6]{lawson}. For a spectral space $X$,  $X^{\mbox{\tiny{{\texttt{inv}}}}}$ coincides with the de Groot dual $X^{\mbox{\tiny{{\texttt{d}}}}}$ (Lemma \ref{uno}(2)). 
\end{oss}

We collect in the following remark some observations concerning Theorem \ref{embedding}.

\begin{oss}  We preserve the notation of Theorem \ref{embedding}.

{\rm (a)} The subspace $\varphi(X)$ is dense in $\boldsymbol{\mathcal{X}}(X)$. 
In fact, let $ \boldsymbol{\mathcal{U}}$ be a nonempty open subset of $\boldsymbol{\mathcal{X}}(X)$, take an element $C\in\boldsymbol{\mathcal{U}}$ and a quasi-compact open subspace $\Omega$ of $X$ such that $C\in \boldsymbol{\mathcal{U}}(\Omega)\subseteq \boldsymbol{\mathcal{U}}$. If $c\in C$, then $\{ c\}^{\mbox{\tiny{{\texttt{gen}}}}}\subseteq C\subseteq \Omega$, and thus $\{c\}^{\mbox{\tiny{{\texttt{gen}}}}}\in \boldsymbol{\mathcal{U}}(\Omega)\subseteq \boldsymbol{\mathcal{U}}$. This proves that $\varphi(X)\cap \Omega\neq \emptyset$.

{\rm (b)} Following \cite[D\'efinition (2.6.3)]{EGA}, recall that a subset $X_0$ of a topological space $X$ is called to be \emph{very dense in $X$} if, for any open sets $U,V\subseteq X$, the equality $U\cap X_0=V\cap X_0$ implies $U=V$.  

The subspace $\varphi(X)$ is not very dense in $\boldsymbol{\mathcal{X}}(X)$.
Indeed,  let $V_1,V_2$ be two discrete rank-one valuation domains having the same quotient field. Then,  the prime spectrum $X$ of the ring $D:=V_1\cap V_2$ consists exactly of $(0)$ and of the two maximal ideals  $M_1$ and $M_2$ which are the  (incomparable) contractions   in $D$ of the maximal ideals of $V_1$ and $V_2$.
Then, in the present situation, 
$$
\begin{array}{rl}
\boldsymbol{\mathcal{X}}(X)=&\{\{(0) \},\{(0),M_1\},\{(0), M_2\},X   \}\,; \\
\varphi(X) =& \hskip -4pt \{\{(0) \},\{(0),M_1\},\{(0), M_2\}\} \,.  
\end{array} 
$$
 Since $\{X\}$ is closed in $\boldsymbol{\mathcal{X}}(X)$, it follows that $\varphi(X)$ is open in $\boldsymbol{\mathcal{X}}(X)$. From this fact, we deduce immediately that $\varphi(X)$ is dense but not very dense in $\boldsymbol{\mathcal{X}}(X)$.

{\rm (c)} Let $X$ be a spectral space and let $\boldsymbol{\hat{\mathcal{X}}}(X):= \boldsymbol{\hat{\mathcal{X}}}:= \{ Y \subseteq X \mid  Y = {\chiusinv}(Y)\} =\boldsymbol{\ \mathcal{X}}(X) \cup \{\emptyset\}$. 
Note that the techniques used in the proof of Theorem \ref{embedding}(1) allow also to show that $\boldsymbol{\hat{\mathcal{X}}}$ (endowed with an obvious extension of the topology of 
$\boldsymbol{ \mathcal{X}}$) is a spectral space.
Since $\ucal(\emptyset)=\{\emptyset\}$ is open in $\boldsymbol{\hat{\mathcal{X}}}$, then 
$\boldsymbol{\mathcal{X}}$ is a closed (spectral) subspace of  $\boldsymbol{\hat{\mathcal{X}}}$.

\end{oss}

\medskip
 
Before stating next result, we observe that $X \in \varphi(X)$ if and only if $X$ has a unique closed point (in the given spectral topology).

\begin{prop}\label{homeo}
Let $X$ be a spectral space and let   $\varphi:X \rightarrow \boldsymbol{\mathcal{X}}(X)$  the topological embedding defined in Theorem \ref{embedding}(3). Then,  $\varphi(X)=   \boldsymbol{\mathcal{X}}(X)$ if and only if $(X,\leq)$ is linearly ordered.
\end{prop}
\begin{proof}   Set, as usual $\boldsymbol{\mathcal{X}} :=\boldsymbol{\mathcal{X}}(X)$.  In order to avoid the trivial case, we can assume that $X$ is not a singleton. 
First, suppose that $(X,\leq)$ is linearly ordered, and let $Y\in \boldsymbol{\mathcal{X}}$. Consider the collection $\mathscr C :=\{\chius(\{y\})\cap Y \mid y\in Y \}$ of closed sets of $Y$ (with respect to the subspace topology induced by the given topology of $X$). Since $(X,\leq)$ is linearly ordered, $\mathscr C$ has the finite intersection property. 
On the other hand, 
  $Y$ is a quasi-compact subspace of $X$, since, in particular,  it is  closed in the constructible topology of $X$ and so it is quasi-compact in the constructible topology and, {\sl a fortiori}, in the given topology of $X$. Thus, there is a point  $y_0\in\bigcap \{ C \mid C \in \mathscr C\}$. Now, it is easy to infer that 
$Y=\{y_0\}^{\mbox{\tiny{{\texttt{gen}}}}}$. 

Conversely, assume that  $\varphi(X)= \boldsymbol{\mathcal{X}}$, 
and take two points $x,y\in X$. 
Clearly, the set   $Z:=\{x,y\}^{\mbox{\tiny{{\texttt{gen}}}}}=\{x\}^{\mbox{\tiny{{\texttt{gen}}}}}\cup\{y\}^{\mbox{\tiny{{\texttt{gen}}}}}$ 
is nonempty and closed with respect to the inverse topology on $X$, i.e.,  $Z \in \boldsymbol{\mathcal{X}}$.   
By assumption, there is a point $z\in X$ such that $\varphi(z)=\{z\}^{\mbox{\tiny{{\texttt{gen}}}}} =\{x\}^{\mbox{\tiny{{\texttt{gen}}}}}\cup\{y\}^{\mbox{\tiny{{\texttt{gen}}}}}$. The inclusion $\supseteq$ implies $x,y\leq z$. On the other hand, the inclusion $\subseteq$ implies that  $z\leq x$ or $z\leq y$. From these facts it follows easily that  $(X,\leq)$ is linearly ordered. 
\end{proof}
We compare next the dimensions of $X$ and $\xcal(X)$ with the cardinality $|X|$ of the spectral space $X$.

\begin{prop}
Let $X$ be a spectral space. Then, $\dim(\xcal(X))=|X|-1\geq\dim(X)$. Moreover, in the finite dimensional case, $\dim(\xcal(X))$ $=\dim(X)$ if and only if $X$ is linearly ordered. 
\end{prop}
\begin{proof}
Suppose first that $X$ is finite. 
If $Y_0 <_{\xcal(X)} Y_1 <_{\xcal(X)} \cdots<_{\xcal(X)} Y_n$ is a chain of points in $\xcal(X)$, then 
$Y_0 \subsetneq Y_1 \subsetneq  \cdots \subsetneq Y_n$ is a chain of  nonempty subsets of $X$. In particular, $|Y_{k-1}| < |Y_k|$ for all $k$, $1\leq k \leq n$. Therefore, $n+1 \leq|X|$   and $\dim(\xcal(X)) \leq|X|-1$.

On the other hand, we can write $X$ as a sequence $x_1, x_2,\ldots,x_t$ (where $t:=|X|$) such that $x_i$ is not bigger that $x_j$ for every  $i<j$ (simply, take $x_1$ as a minimal element of  $X$ and $x_i$ as a minimal element of $X\setminus\{x_1,\ldots,x_{i-1}\}$ for $i \geq 2$).
 In particular, each $X_i:=\{x_1, x_2,\ldots,x_i\}$ is inverse-closed in $X$, so that $X_1<_{\xcal(X)} X_2<_{\xcal(X)} \cdots <_{\xcal(X)}  X_t$ is a chain of points in $\xcal(X)$ of length $t-1$.  Therefore, $\dim(\xcal(X))\geq|X|-1$  and, by the previous paragraph, we conclude that $\dim(\xcal(X))=|X|-1$.

Suppose now that $X$ is infinite.  Take a positive integer $t$ and let $X'$ be a subset of $X$ of cardinality $t$. As before, we can enumerate the elements $x_1, x_2, \ldots, x_t$ of $X'$ in such a way that 
$x_i$ is not bigger that $x_j$ for every  $i<j$.
Then, for each $i\in\{1,2, \ldots,t\}$, the set $C_i:=\{x_1, x_2, \ldots, x_i \}^{\mbox{\tiny{{\texttt{gen}}}}}$ is closed in the inverse topology of $X$, i.e., $C_i\in \xcal(X)$. Clearly, $C_i \subsetneq C_{i+1}$ for each $i=1,2, \ldots,t-1$, since $x_{i+1}\in C_{i+1}\setminus C_i$. This proves that, for any positive integer $t$, there is a chain of lenght $t-1$ in $\xcal(X)$. Thus, $\dim(\xcal(X))=\infty$. 

If $X$ is finite,  $\dim(X) = |X| -1$ if and only if, in $X$, there is a chain of the type $x_0 <x_1< \dots  x_{|X|-1}$. This means that all elements of $X$ are in such chain, i.e., $X$ is linearly ordered.
\end{proof}

\begin{oss}
While the inequality $|X|-1\geq\dim(X)$ is sharp by the previous proposition, 
the more non-comparable elements the set $X$ contains the more $\dim(X)$ is small with respect to $|X|$. For example, if $X$ is homeomorphic to the prime spectrum of the direct product of $n+1$ fields, $n\geq 1$, then $\dim(X)=0$ while $|X|-1=n$.  

If $\dim(X)$ is not finite, then clearly $\dim(\xcal(X))=\dim(X)$, but we can easily choose $X$  to be not linearly ordered. 
\end{oss}

\section{Functorial properties}
A map $\psi:X_1\rightarrow X_2$ of spectral spaces is called \emph{spectral} if $\psi^{-1}(\Omega)$ is a quasi-compact open subset of $X_1$ for every quasi-compact open subset $\Omega$ of $X_2$.

\begin{prop}\label{prop:appl}
Let $\psi:X_1\rightarrow X_2$ be a spectral map of spectral spaces and denote by  $\varphi_1:  X_1 \rightarrow\xcal(X_1)$ and $\varphi_2:  X_2 \rightarrow\xcal(X_2)$  the topological embeddings defined in Theorem \ref{embedding}(3).  Then, there is a spectral map $\xcal(\psi):\xcal(X_1)\rightarrow\xcal(X_2)$ such that $\xcal(\psi)\circ \varphi_1 = \varphi_2\circ \psi$.
\end{prop}
\begin{proof}
 First note that each $C\in\xcal(X_1)$ is quasi-compact in $X_1$ and so $\psi( C)$ is quasi-compact in $X_2$ and thus $\chiusinv(\psi( C)) = \psi( C)^{\mbox{\tiny{{\texttt{gen}}}}} =\bigcup\{\{x_2\}^{\mbox{\tiny{{\texttt{gen}}}}}  \mid x_2 \in \psi(C)\} =\sup\{\{x_2\}^{\mbox{\tiny{{\texttt{gen}}}}}  \mid x_2 \in \psi( C)\}$ \cite[Remark 2.2 and Proposition 2.6]{fifolo2}.
For every $C\in\xcal(X_1)$, define $\xcal(\psi)( C):=\psi( C)^{\mbox{\tiny{{\texttt{gen}}}}}$. 
In particular, we have that
 $\xcal(\psi)(\{x\}^{\mbox{\tiny{{\texttt{gen}}}}})= 
 \{\psi(x)\}^{\mbox{\tiny{{\texttt{gen}}}}}$, for each $x \in X_1$.

Let  $\Omega$ be a quasi-compact open subset of $X_2$. We claim that 
$$
 (\xcal(\psi))^{-1}(\ucal(\Omega)) =\ucal(\psi^{-1}(\Omega)),
$$ 
which is quasi-compact open in $\xcal(X_2)$, since $\psi$ is spectral (and so $\psi^{-1}(\Omega)$ is quasi-compact open in $X_1$). 
As a matter of fact, let $C \in (\xcal(\psi))^{-1}(\ucal(\Omega))$, i.e., $\xcal(\psi)( C) \subseteq \Omega$, therefore  $\psi^{-1}(\xcal(\psi)( C)) \subseteq  \psi^{-1}(\Omega)$ and thus, clearly, $C \subseteq$ $ \psi^{-1}(\xcal(\psi)( C)) $.
Conversely, let $C \subseteq \psi^{-1}(\Omega)$, then $ \xcal(\psi)( C) \leq \xcal(\psi)( \psi^{-1}(\Omega))$.
  Moreover, we have that $\xcal(\psi)( \psi^{-1}(\Omega))$ $= 
  (\psi( \psi^{-1}(\Omega)))^{\mbox{\tiny{{\texttt{gen}}}}} \subseteq \Omega^{\mbox{\tiny{{\texttt{gen}}}}} =\Omega$.
  Therefore $ \xcal(\psi)( C) \in \ucal(\Omega)$. 
  We conclude that $\xcal(\psi)$ is a spectral map. 
\end{proof}

It is well known that, for compact Hausdorff spaces and, hence, for Stone spaces, the upper Vietoris construction is functorial. Si\-mi\-larly, we now show that
the assignment $\xcal$ defined by the pair $(X\mapsto\xcal(X), \psi\mapsto\xcal(\psi))$ can be interpreted as a functor from the category of spectral spaces into itself.

\begin{prop}\label{prop:restrizPsi}
We preserve the notation of  Proposition \ref{prop:appl}.

\begin{itemize}
\item[(1)]
If $X_1\xrightarrow{\psi_1}X_2\xrightarrow{\psi_2}X_3$ is a chain of spectral maps, then the spectral map $\xcal(\psi_2\circ\psi_1):  \xcal(X_1) \rightarrow\xcal(X_3)$, induced by $\psi_2\circ\psi_1$ is equal to the composition  $\xcal(\psi_2) \circ \xcal(\psi_1)$. It follows that the assignment $(X\mapsto\xcal(X), \psi\mapsto\xcal(\psi))$  defines a functor from the category of spectral spaces into itself. 
\item[(2)] 
Let $\Psi:\xcal(X_1)\rightarrow\xcal(X_2)$ be a spectral map.
Assume that there exists a spectral map $\psi:X_1\rightarrow X_2$ such that $\Psi \circ \varphi_1= \varphi_2 \circ \psi$, then, $\xcal(\psi)\leq\Psi$ (i.e., $\xcal(\psi)( C) \subseteq \Psi( C)$ for each $C\in \xcal(X_1))$.
\end{itemize}
\end{prop}
\begin{proof}
(1) The proof is straightforward.

(2) Let $C\in\xcal(X_1)$. For every $c\in C$, we have  $C\supseteq \varphi_1( c) = \{c\}^{\mbox{\tiny{{\texttt{gen}}}}} $ (i.e., $C\geq\varphi_1( c)$ with respect to the order of $\xcal(X_1)$ induced by the Zariski topology).
 Since $\Psi$ is continuous, it is order-preserving, and thus $\Psi( C)\geq\Psi(\varphi_1( c))= \varphi_2(\psi( c)) = \{\psi( c)\}^{\mbox{\tiny{{\texttt{gen}}}}} $. Hence, $\psi( c)\in\Psi( C)$, and thus $\psi( C)\subseteq\Psi( C)$. 
 Since $\Psi( C)$ is closed in the inverse topology on $X_2$, then $\chiusinv(\psi( C))\subseteq\Psi( C)$. On the other hand, by definition, $\xcal(\psi)( C)=
  \psi( C)^{\mbox{\tiny{{\texttt{gen}}}}}  =\chiusinv(\psi( C)) 
  \leq \Psi( C)$, hence $\xcal(\psi)\leq\Psi$.
\end{proof}

\begin{oss}
The previous result is very similar to   the statement concerning the functoriality of the Smyth powerdomain construction $\boldsymbol{\mathcal Q}(X)$ proved in \cite[Proposition IV.8.19, page 371]{GHKLMS} when $X$ is a \emph{directed-complete partial order} (that is, a partially ordered set where each directed subset has a supremum) endowed with the topology generated by the upper sets (called the \emph{Scott topology}).
 However, despite the similarity of the construction, the Scott topology does not coincide with the given spectral topology but, in general, it is stronger than the inverse topology \cite[Proposition 2.9]{HK}. Nevertheless, by order-theoretic reasons, the functoriality of the Smyth powerdomain construction $\boldsymbol{\mathcal Q}(X)$ given in \cite{GHKLMS} is closer to functoriality of the construction $\xcal'(X):=\xcal(X^{\mbox{\tiny{{\texttt{inv}}}}})$ \cite{fi-fo-sp-MANUSCRIPTA} recalled briefly in Remark \ref{oss-inv}(a).  
\end{oss}


The next example shows that it is possible to have $\Psi\neq\xcal(\psi)$, i.e., it is possible to have more than one ``extension'' of $\psi:X_1\rightarrow X_2$  between the spaces $\xcal(X_1)$ and $\xcal(X_2)$.  On the other side, we will show in the following Proposition \ref{psi} that this situation does not occur when $\Psi$ is a homeomorphism.
 
\begin{ex} \label{ex-ext}
Let $X_1=\{a_1,a_2, b\}$ and $X_2:=\{c_1, c_2\}$. Suppose that $a_1$ and $a_2$ are incomparable but both smaller than $b$ and suppose also that $c_1< c_2$.   It is straightforward that the order structures of $X_1$ and $X_2$ are compatible with the order of suitable spectral topologies on $X_1$ and $X_2$.  When $X_1$ and $X_2$ are equipped with these spectral topologies, it is easy to see that $\xcal(X_1)=\{\{a_1\},\{a_2\},\{a_1,a_2\},\{b,a_1,a_2\}\}$, while $\xcal(X_2)=\{\{c_1\},\{c_1,c_2\}\}$. Let $\psi:X_1\rightarrow X_2$  be the spectral map defined by $\psi(a_1):=\psi(a_2):=c_1$ and $\psi(b):=c_2$.   Let $\Psi:\xcal(X_1)\rightarrow \xcal(X_2)$ be the map defined by 
 $\Psi(\{a_1\}):=\Psi(\{a_2\}):=\{c_1\}$ and $\Psi(\{b,a_1,a_2\}):=\Psi(\{a_1,a_2\}):=\{c_1,c_2\}$. Clearly, $\Psi$ is a spectral map of spectral spaces, since 
 $$
 \Psi^{-1}(\ucal(\{c_1\}))=\{\{a_1\},\{a_2\} \}=\ucal(\{a_1\})\cup \ucal(\{a_2\}),
 $$
 and  $\Psi^{-1}(\ucal(\{c_1, c_2\}))= \Psi^{-1}(X_2)= X_1$.   Moreover, it is obvious that $\Psi$ ``extends'' $\psi$. 
 However,  the ``natural extension'' $\xcal(\psi)$ of $\psi$ (defined in Proposition \ref{prop:appl}) is such that     $\xcal(\psi)(\{a_1,a_2\})=\{c_1\}$, and thus $\Psi\neq \xcal(\psi)$. The situation is illustrated in Figure \ref{fig:Psi-xcalpsi}.

\begin{figure}
\includegraphics[scale=0.7]{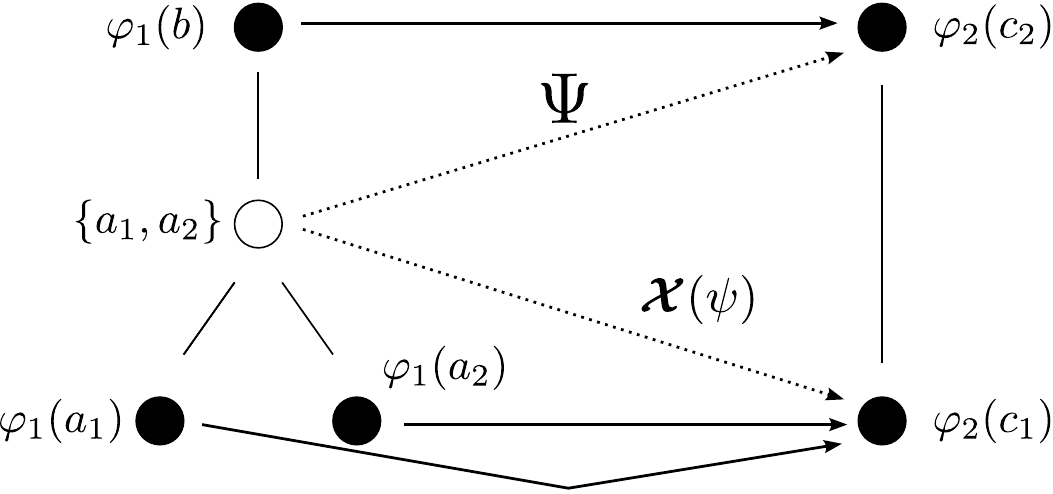}
\caption{Illustration of   Example \ref{ex-ext}. Black circle represents elements of $\varphi_1(X_1)$ and $\varphi_2(X_2)$.}
\label{fig:Psi-xcalpsi}
\end{figure}
\end{ex}

\begin{prop} \label{psi}
Let $X_1,X_2$ be spectral spaces and let    $\varphi_1:X_1\rightarrow\xcal(X_1)$ and $\varphi_2:X_2\rightarrow\xcal(X_2)$ be the canonical embeddings (as in Theorem \ref{embedding}(3)).
\begin{itemize}
\item[(1)]
If $\psi:X_1\rightarrow X_2$ is a topological embedding (respectively, an homeomorphism),  then  $\xcal(\psi):\xcal(X_1)\rightarrow\xcal(X_2)$  (as defined in Proposition \ref{prop:appl}) is a topological embedding (resp., homeomorphism).
\item[(2)] If $\Psi:\xcal(X_1)\rightarrow\xcal(X_2)$ is a homeomorphism, then there exists a   unique  homeomorphism $\psi:
X_1\rightarrow X_2$ such that $\Psi = \xcal(\psi)$ (and so  $\Psi \circ \varphi_1= \varphi_2 \circ \psi$).
\item [(3)]   In  particular, $X_1$ and $X_2$ are homeomorphic if and only if $\xcal(X_1)$ and $\xcal (X_2)$ are homeomorphic. 
\end{itemize} 
\end{prop} 
\begin{proof}
(1) By Proposition \ref{prop:appl}, $\xcal(\psi)\circ\varphi_1=\varphi_2\circ\psi$. Since $\varphi_1$ and $\varphi_2$ are topological embeddings, if also $\psi$ is an embedding so is $\varphi_2\circ\psi$, and thus so is $\xcal(\psi)\circ\varphi_1$; hence also $\xcal(\psi)$ is an embedding. If $\psi$ is an homemorphism, and $C\in\xcal(X_2)$, then $C=\xcal(\psi)(\psi^{-1}(C))$, so that $\xcal(\psi)$ is surjective and thus an homeomorphism.

(2)    We start by showing the following.

\textbf{Claim 1.} \textit{Let $X$ be a spectral space and let $\varphi:X\rightarrow \xcal(X)$ be the canonical embedding. Then, $\varphi(X)$ is precisely the set of all the irreducibile closed subset of $X$, endowed with the inverse topology.} 

As a matter of fact,   it is well known that the space $X^{\mbox{\tiny{{\texttt{inv}}}}}$, i.e., the set $X$ endowed with the inverse topology,  is itself a spectral space \cite[Proposition 8]{ho}, and thus any irreducible closed subspace $C$ of $X^{\mbox{\tiny{{\texttt{inv}}}}}$ has a unique generic point, say $x$, that is $C=\chius^{\mbox{\tiny{{\texttt{inv}}}}}(\{x\})=\{x\}^{\mbox{\tiny{{\texttt{gen}}}}}=\varphi(x)$. 
On the orther hand, it is trivial that $\varphi(X)$ is contained in  the set of all the irreducibile closed subset of $X^{\mbox{\tiny{{\texttt{inv}}}}}$.

\textbf{Claim 2.} \textit{Assume that  $\Psi:\xcal
(X_1)\rightarrow \xcal(X_2)$ is a homeomorphism. Let $C$ be an irreducible and closed subspace of $X_1^{\mbox{\tiny{{\emph{\texttt{inv}}}}}}$. Then $\Psi( C)$ is an irreducible (and closed) subset of $X_2^{\mbox{\tiny{{\emph{\texttt{inv}}}}}}$.}  

Let $D' , D'' \in \xcal(X_2)$ be such that $D'\cup D''=\Psi( C)$. Since $\Psi$ is a homeomorphism is also an isomorphism of ordered sets, 
 we see that  $C=\Psi^{-1}(D')\cup \Psi^{-1}(D'')$. Since $C$ is irreducible, we have either $C=\Psi^{-1}(D')$ or $C=\Psi^{-1}(D'')$, and thus either $\Psi( C)=D'$ or $\Psi( C)=D''$. 

 Now, fix a point $x\in X_1$. By Claim  2, the set $\Psi(\{x\}^{{\mbox{\tiny{{\texttt{gen}}}}}})$ 
 is irreducible in $X_2$, thus by Claim 1 there is a unique point $x_\Psi\in X_2$ such that 
 $\{x_\Psi \}^{\mbox{\tiny{{\texttt{gen}}}}}=
 \Psi(\{x\}^{\mbox{\tiny{{\texttt{gen}}}}})$. 
Thus $\Psi$ induces naturally a map $\psi:X_1\longrightarrow X_2$ by setting $\psi(x):=x_\Psi$, for any $x\in X$. Clearly, $\varphi_2\circ \psi=\Psi\circ \varphi_1$. 
Next, we want to who that $\psi:X_1\rightarrow X_2$ is homeomorphism.

\textbf{Claim 3.} \textit{Assume that  $\Psi:\xcal
(X_1)\rightarrow \xcal(X_2)$ is a homeomorphism. Let $\Omega$ be a quasi-compact open subspace of $X_1$ (in particular, $\Omega \in \xcal
(X_1)$).  Then $\Psi( \Omega)$ is a quasi-compact open subspace of $X_2$.}  

Note that  the quasi-compact open subspace $\ucal(\Omega)$ of $\xcal
(X_1)$ coincides with $\{\Omega\}^{{\mbox{\tiny{{\texttt{gen}}}}}}$ (where the generizations are taken in $\xcal
(X_1))$. Since $\Psi$ is a homeomorphism, then $\Psi(\ucal(\Omega)) =\Psi(\{\Omega\}^{{\mbox{\tiny{{\texttt{gen}}}}}}) = 
\{\Psi(\Omega)\}^{{\mbox{\tiny{{\texttt{gen}}}}}}$ is a quasi-compact open set of 
$\xcal(X_2)$ which is irreducible as inverse-closed subspace of 
$\xcal(X_2)$. 
In order to show that $\Psi( \Omega)$ is a quasi-compact open subspace of $X_2$,
 we observe that $ \Psi(\{\Omega\}^{{\mbox{\tiny{{\texttt{gen}}}}}})=
 \Psi(\ucal(\Omega))=
 \bigcup\{ \ucal(V_i) \mid 1\leq i \leq n\} =
  \bigcup\{ \{V_i\}^{{\mbox{\tiny{\emph{\texttt{gen}}}}}} \mid 1\leq i \leq n\}$, 
  for a finite family of quasi-compact open subspaces   
  $\{  V_i  \mid 1\leq i \leq n\}$ of $X_2$. 
  Therefore, $ \Psi(\{\Omega\}^{{\mbox{\tiny{{\texttt{gen}}}}}}) =  
  \{V_{\tilde{i}}\}^{{\mbox{\tiny{{\texttt{gen}}}}}}$
   for some $\tilde{i}$ and so $ \Psi(\Omega) =  V_{\tilde{i}}$.

In order to prove that $\psi: X_1 \rightarrow X_2$ is a homeomorphism, we start by showing that $\psi$ is continuous.
Let $V\subseteq X_2$ be a quasi-compact open. We claim that $\psi^{-1}(V)=\Psi^{-1}(V)$, where $\Psi^{-1}(V)\in\xcal(X_1)$ is a quasi-compact open subspace of $X_1$, since $\Psi$ is a homeomorphism. 
Moreover, $\ucal(\Psi^{-1}(V)) = \{\Psi^{-1}(V)\}^{{\mbox{\tiny{{\texttt{gen}}}}}} = \Psi^{-1}(\{(V)\}^{{\mbox{\tiny{{\texttt{gen}}}}}}) =
 \Psi^{-1}(\ucal(V)) $.
Now, take a point $x\in  X_1$. Then
$$
\begin{array}{rl}
\psi(x)\in V \Leftrightarrow &  \{x_\psi\}^{{\mbox{\tiny{{\texttt{gen}}}}}}\subseteq V \, \Leftrightarrow \, \Psi(\{x\}^{{\mbox{\tiny{{\texttt{gen}}}}}})\in \ucal(V)\\
\Leftrightarrow & \{x\}^{{\mbox{\tiny{{\texttt{gen}}}}}}\in \Psi^{-1}(\ucal(V)) \Leftrightarrow x\in \Psi^{-1}(V),
\end{array}
$$
i.e., $\psi^{-1}(V) = \Psi^{-1}(V)$.

Now, we show that  $\psi: X_1 \rightarrow X_2$ is open.
Let $\Omega$ be a quasi-compact open subspace of $X_1$. By Claim 3, $\Psi(\Omega)$ is a quasi-compact open subspace of $X_2$ and, obviously, $\Omega=\Psi^{-1}(\Psi(\Omega))$. Moreover, by the previous observation, $\Psi^{-1}(\Psi(\Omega))= \psi^{-1}(\Psi(\Omega))$ and, since $\psi$ is bijective, $\psi(\Omega)=\Psi(\Omega)$.

  Finally, we show that $\xcal(\psi)=\Psi$. Take a set $C\in \xcal(X_1)$. Since $\psi$ is a homeomorphism, it is also a homeomorphism between $X_1^{\rm inv}$ and $X_2^{{\mbox{\tiny{{\texttt{inv}}}}}}$, and in particular it is a closed map (with respect to the inverse topologies).
Therefore, it suffices to prove that $\xcal(\psi)( C)  = 
\psi( C)^{{\mbox{\tiny{{\texttt{gen}}}}}} = \psi( C)$
coincides with $\Psi(C)$. 
Let $\{C_i \mid i\in I \}$ be the collection of the irreducible (and closed) components of $C$  in $X_1^{\rm inv}$. By Claim 1, for any $i\in I$, let $x_i\in X_1$ be the unique generic point of $C_i$ in $X_1^{{\mbox{\tiny{{\texttt{inv}}}}}}$. Keeping in mind that both $\psi$ and $\Psi$ are also isomorphism of partially ordered sets (orderings induced by the topologies), we have
$$
\begin{array}{rl}
\Psi(C)= & \hskip -5 pt \Psi(\sup\{C_i\mid i\in I\})=\sup\{\Psi(C_i)\mid i\in I\}= \\ = & \hskip -5 pt 
 \sup\{\Psi(\{x_i\}^{{\mbox{\tiny{{\texttt{gen}}}}}}) \mid i\in I\}= \\
=&\hskip -5 pt\sup \{ \{\psi(x_i)\}^{{\mbox{\tiny{{\texttt{gen}}}}}} \mid   i\in I \} 
= \bigcup \{\{\psi(x_i)\}^{{\mbox{\tiny{{\texttt{gen}}}}}} \mid \ i\in I\} =
\\ =&\hskip -5 pt
\psi(\bigcup \{\{x_i\}^{{\mbox{\tiny{{\texttt{gen}}}}}}\mid i\in I\})=\psi(\bigcup\{C_i \mid i\in I\})=\psi(C) \,.
\end{array}
$$
The proof of (2) is now complete.
 Part (3) is an immediate consequence of statements (1) and (2). 
\end{proof}

\medskip

 It is not difficult to see that $\varphi (= \varphi_X): X \rightarrow \xcal(X)$ does not provide a unique way  for embedding a spectral space $X$ in a larger ``natural'' spectral space.  However, $\varphi$  satisfies an universal-like property. 

We start with a lemma.
 
\begin{lemma}\label{sup-spectral}
Let $Z$ be a spectral space and let $Y$ be a closed set   in the constructible topology  of $Z$; in particular, $Y$ is a spectral space.  Assume that the map 
${\large\mbox{$\Sigma$}}_{Y,Z}\colon\xcal(Y) \rightarrow Z$, 
$C  \mapsto {\textstyle\sup}_Z(C)$, for each $C \in \xcal(Y)$,
is well-defined. Then, the following statements hold.
\begin{enumerate}
\item[{\rm (1)}] If each point of $Z$ has a local basis consisting of sets of the form $\{\omega \}^{{\mbox{\tiny{\emph{\texttt{gen}}}}}}$, for suitable elements 
$\omega \in Z$, then $\Sigma_{Y,Z}$ is continuous, spectral  and open onto its image.
\item[{\rm (2)}] If $Y=Z$, then the converse hold.
\end{enumerate}
\end{lemma}
\begin{proof} 
(1) For the sake of simplicity, set $\Sigma:=\Sigma_{Y,Z}$. Let $x\in Z$   and $V_x$ a basic open set of $Z$ containing $x$: then, we claim that $\Sigma^{-1}(V_x)=\ucal(V_x\cap Y)$ and that $\Sigma(\ucal(V_x\cap Y))=V_x\cap\Sigma(\xcal(Y))$. (Note that, since $Y$ is closed, with respect to the constructible topology, $V_x\cap Y$ is open in $Y$ and quasi-compact, and thus determines a basic open set of $\xcal(Y)$.) Indeed, take a point $K\in\Sigma^{-1}(V_x)$: then $k:=\sup_Z(K)\in V_x$, and thus $K\subseteq \{k \}^{{\mbox{\tiny{{\texttt{gen}}}}}}\subseteq V_x$. Since, clearly, $K\subseteq Y$, we have  $K\in \ucal(V_x\cap Y)$. Conversely, take a point $K\in \ucal(V_x\cap Y)$: in particular, $K\subseteq  V_x$, and thus we have $k:=\sup_Z(K)\leq x$, or equivalently $k\in V_x$. Hence, $K\in\Sigma^{-1}(V_x)$. This reasoning also shows the second equality.

The hypotheses on $Z$ now imply that $\Sigma$ is continuous, spectral and open onto its image.

(2)   Now, let $\Sigma:=\Sigma_{Z,Z}$. 
Take a point $z\in Z$ and an open neighborhood $U$ of $Z$. Since $z=\Sigma(\{z\}^{{\mbox{\tiny{{\texttt{gen}}}}}})$ and $\Sigma$ is continuous, there is a quasi-compact open subspace $\Omega$ of $Z$ such that $\{z\}^{{\mbox{\tiny{{\texttt{gen}}}}}}\in \ucal(\Omega)$ (i.e., $z\in\Omega$) and $\Sigma(\ucal(\Omega))\subseteq U$. Since $\Omega\in \ucal(\Omega)$, the last statement implies $\omega:=\sup_Z(\Omega)\in U$. It follows $z\in \Omega\subseteq \{\omega\}^{{\mbox{\tiny{{\texttt{gen}}}}}}\subseteq U$. 
\end{proof}

\begin{oss} Let $Z$ be a spectral space and let $\varphi_Z: Z\rightarrow \boldsymbol{\mathcal{X}}(Z)$
 be the spectral embedding  introduced in Theorem \ref{embedding}(3). 
 Under the assumptions and the equi\-valent conditions of Lemma \ref{sup-spectral}, the map $\Sigma_Z$ (= $\Sigma_{Z,Z}$)  gives rise to a topological retraction, since $\Sigma_Z \circ\ \varphi_Z$ is the identity map   on $Z$.
\end{oss}

\smallskip

We say that a map $f:X\rightarrow Y$ of spectral spaces is \emph{sup-preserving} if, whenever $F$ is a finite subset of $X$ and there exists $\sup_X(F)$, then there exists $\sup_Y(f(F))$ and $f(\sup_X(F))=\sup_Y(f(F))$.

\smallskip

\begin{teor} \label{sigma}
Let $X$ be a spectral space and let $\varphi (= \varphi_X):X\rightarrow\xcal(X)$ be the canonical spectral embedding (Theorem \ref{embedding}(3)). Let $Z$ be a spectral space, and suppose that the map {\large\mbox{$\Sigma$}} (= {\large\mbox{$\Sigma$}}${_{\lambda(X),Z})}:\xcal(\lambda(X))\rightarrow Z$, introduced in Lemma \ref{sup-spectral}, is (well-defined and) spectral. Let $\lambda:X\longrightarrow Z$ be a spectral map.
\begin{enumerate}
\item[{\rm (1)}]  There is a sup-preserving spectral map ${\boldsymbol\lambda^\sharp}:\xcal(X)\rightarrow Z$, defined by setting ${\boldsymbol\lambda^\sharp}(C):=\sup_Z(\lambda(C)^{{\mbox{\tiny{\emph{\texttt{gen}}}}}})$, for each $C\in\xcal(X)$, such that ${\boldsymbol\lambda^\sharp}\circ \varphi =\lambda$.
\item[{\rm (2)}]  If $\boldsymbol{\Lambda}:\xcal(X)\rightarrow Z$ is a spectral map such that ${\boldsymbol\Lambda}\circ \varphi =\lambda$, then ${\boldsymbol\lambda^\sharp}(K)\leq\boldsymbol{\Lambda}(K)$ for every $K\in\xcal(X)$ (where $\leq$ is the order induced on $Z$ by the topology).
\item[{\rm (3)}]  If, moveover, ${\boldsymbol\Lambda}$ is sup-preserving, then ${\boldsymbol\Lambda}={\boldsymbol\lambda^\sharp}$.
\end{enumerate}
\end{teor}
\begin{proof}
(1) Since $\lambda$ is a spectral map, it is also continuous when $X$ and $Z$ are both endowed with the constructible topology. In particular, since the constructible topology is both quasi-compact and Hausdorff, $\lambda$ is a closed map when considered in the constructible topology, and thus $\lambda(X)$ is a closed set in the constructible topology of $Z$; therefore, $\lambda(X)$ is a spectral space (so that $\xcal(\lambda(X))$ is well-defined) and the inclusion $j:\lambda(X)\hookrightarrow Z$ is a spectral map. In particular, it is possible to define the map 
${\mbox{\large$\Sigma$}} (= {\mbox{\large $\Sigma$}}{_{\lambda(X),Z}})$.

Let ${\boldsymbol\lambda^\sharp}:\xcal(X)\rightarrow Z$ be the map defined above.
 Keeping in mind \cite[Propositions 2.1 and 2.2]{olb2015} and the fact that any point of a quasi-compact T$_0$ space is $\leq$ than a maximal point of the space, we easily infer that
  ${\boldsymbol\lambda^\sharp}$ = {\large\mbox{$\Sigma$}}$ \circ \xcal(\lambda)$ and thus, by assumption, ${\boldsymbol\lambda^\sharp}$ is spectral. 
Moreover, both {\large\mbox{$\Sigma$}} and $\xcal(\lambda)$ are sup-preserving (easy verification), and thus ${\boldsymbol\lambda^\sharp}$ is sup-preserving; by definition it follows that ${\boldsymbol\lambda^\sharp} \circ \varphi=\lambda$. 

 (2) Suppose now that $\boldsymbol{\Lambda}:\xcal(X)\rightarrow Z$ is such that $\boldsymbol{\Lambda}\circ\varphi=\lambda$, and fix a $K\in\xcal(X)$.

For each $x\in K$ we have 
$\{x\}^{{\mbox{\tiny{{\texttt{gen}}}}}}\subseteq K$ and since, in particular,  $\boldsymbol{\Lambda}$ is continuous, it follows   that
\begin{equation*}
\lambda(x)=\boldsymbol{\Lambda}(\varphi(x))=\boldsymbol{\Lambda}(\{x\}^{{\mbox{\tiny{{\texttt{gen}}}}}})\leq \boldsymbol{\Lambda}(K).
\end{equation*}
 By definition, ${\boldsymbol\lambda^\sharp}(K)$ is equal to the supremum in $Z$ of the set $\lambda(K)^{{\mbox{\tiny{{\texttt{gen}}}}}}$; 
 moreover, it is equal to the supremum of $\lambda(K)$, since   if  $y\in\lambda(K)^{{\mbox{\tiny{{\texttt{gen}}}}}}$
  then $y\leq\lambda(x)$ for some $x\in K$. By the previous calculation, $\lambda(x)\leq\boldsymbol{\Lambda}(K)$ for every $x\in K$;   therefore, ${\boldsymbol\lambda^\sharp}(K)\leq\boldsymbol{\Lambda}(K)$, as claimed.

 (3) Suppose now that the spectral map $\boldsymbol{\Lambda}$ is sup-preserving, and as above let $K\in\xcal(X)$.
Take any open neighborhood $V$ of $z:={\boldsymbol\lambda^\sharp}(K)$ in $Z$. 
Then, by definition   and by (2),  in order to prove that $ {\boldsymbol\lambda^\sharp}(K)= \boldsymbol{\Lambda}(K)$, it suffices to show that $\boldsymbol{\Lambda}(K)\in V$. 
Since {\large\mbox{$\Sigma$}} is continuous, then there exist an element $v\in V$ and a quasi-compact open subspace $W$ of $Z$
 such that $z\in W\subseteq \{v\}^{{\mbox{\tiny{{\texttt{gen}}}}}}\subseteq V$, in view of Lemma \ref{sup-spectral}. 
 For any $x\in K$, we have
\begin{equation*}
\boldsymbol{\Lambda}(\{x\}^{{\mbox{\tiny{{\texttt{gen}}}}}})=\boldsymbol{\Lambda}(\varphi(x))=\lambda(x)\leq{\textstyle\sup_Z}(\lambda(K)^{{\mbox{\tiny{{\texttt{gen}}}}}})=z\in W\,.
\end{equation*}
Since $W$ is (in particular) closed under generizations, it follows $\boldsymbol{\Lambda}(\{x\}^{{\mbox{\tiny{{\texttt{gen}}}}}})$ $\in W$. 
Since $\boldsymbol{\Lambda}$ is continuous, there is a quasi-compact open subspace $A_x$ of $Z$ such that
 $\{x\}^{{\mbox{\tiny{{\texttt{gen}}}}}}\in \ucal(A_x)$ (i.e., $x\in A_x$) and $\boldsymbol{\Lambda}(\ucal(A_x))\subseteq W$. 
 Thus, $\bigcup_{x\in K}A_x\supseteq K$ and, since $K$ is (in particular) quasi-compact, there are finitely many elements
  $x_1, x_2,\z,x_n\in K$ such that $K\subseteq \bigcup_{i=1}^nA_{x_i}$. 
  Note that $\bigcup_{i=1}^nA_{x_i}\in \xcal(Z)$, since any $A_{x_i}\in\xcal(Z)$ is open and quasi-compact.
   Keeping in mind that $\boldsymbol{\Lambda}$ is continuous (and thus an 
   order-preserving map), we have
\begin{equation*}
\begin{array}{rl}
\boldsymbol{\Lambda}(K)\leq \boldsymbol{\Lambda}\left(\bigcup_{i=1}^nA_{x_i}\right)=& \hskip -5pt \boldsymbol{\Lambda}\left({\textstyle\sup}_{\xcal(Z)}(\{A_{x_i}\mid 1\leq i\leq n \})\right)\\ 
=&\hskip -5pt {\textstyle\sup}_{Z}\left(\{\boldsymbol{\Lambda}(A_{x_i})\mid 1\leq i\leq n  \}\right).
\end{array}
\end{equation*}
Since $\boldsymbol{\Lambda}(\ucal(A_{x_i}))\in W\subseteq \{v\}^{{\mbox{\tiny{{\texttt{gen}}}}}}$, for $1\leq i\leq n$, it follows that 
$\sup_Z(\{\boldsymbol{\Lambda}(A_{x_i})\mid 1\leq i\leq n  \})\in \{v\}^{{\mbox{\tiny{{\texttt{gen}}}}}}\subseteq V,$
and, a fortiori, $\boldsymbol{\Lambda}(K)\in V$. The proof is now complete. 
\end{proof}

\begin{oss}
The last part of the previous theorem  provides a slight generalization of \cite[Proposition 5.6]{lawson}. Indeed, under the equivalence between the construction $\xcal(X)$ (with the Zariski topology) and the Smyth powerdomain $\boldsymbol{\mathcal {Q}}(X)$ (with the upper Vietoris topology) estabilished in Proposition \ref{zariski-vietoris}, a sup-preserving map becomes an homomorphism of semilattices, and the map {\large\mbox{$\Sigma$}} coincides with the map $\bigwedge$  considered  in \cite{lawson}. The difference between Theorem \ref{sigma} and \cite[Proposition 5.6]{lawson} is that we do not require the map {\large\mbox{$\Sigma$}} to exist on the whole $\xcal(Z)$, but only on $\xcal(\lambda(X))$.
\end{oss} 

\begin{prop}
Preserve the notation and the hypotheses of  Theorem \ref{sigma}, and suppose that the map {\large\mbox{$\Sigma$}}  (= {\large\mbox{$\Sigma$}}$_{\lambda(X),Z}):
\xcal(\lambda(X))\rightarrow Z$ is injective.  Then, the following hold.
\begin{enumerate}
\item[{\rm (1)}]  ${\boldsymbol\lambda^\sharp}$ is a spectral embedding.
\item[{\rm (2)}]  If, furthermore, $z=\sup_Z\{\lambda(x)\mid x\in\lambda^{-1}(\{z\}^{{\mbox{\tiny{\emph{\texttt{gen}}}}}})\}$ for every $z\in Z$, and ${\boldsymbol\Lambda}:\xcal(X)\rightarrow Z$ is a spectral embedding such that ${\boldsymbol\Lambda}\circ\varphi=\lambda$, then ${\boldsymbol\Lambda}={\boldsymbol\lambda^\sharp}$.
\end{enumerate}
\end{prop}
\begin{proof}
(1) The proof of Lemma \ref{sup-spectral} shows that {\large\mbox{$\Sigma$}} is a spectral embedding whenever it is injective. Since $\varphi$ is also a spectral embedding, so is {\large\mbox{$\Sigma$}}$\circ\varphi$, i.e., ${\boldsymbol\lambda^\sharp}$.

(2)  In the present situation, we claim that ${\boldsymbol\Lambda}$ is sup-preserving. Let $C_1,C_2\in\xcal(X)$, and consider ${\boldsymbol\Lambda}(C_1\cup C_2)$ (note that the order on $\xcal(X)$ is   the set-theoretic inclusion, so the union is exactly their supremum); clearly, ${\boldsymbol\Lambda}(C_1\cup C_2)$ is bigger than both ${\boldsymbol\Lambda}(C_1)$ and ${\boldsymbol\Lambda}(C_2)$, and thus also of their supremum.

Let $x$ be such that $\lambda(x)\leq {\boldsymbol\Lambda}(C_1\cup C_2)$, or equivalently such that $x\in\lambda^{-1}({\boldsymbol\Lambda}(C_1\cup C_2))$. Since $\lambda(x)={\boldsymbol\Lambda}(\{x\}^{{\mbox{\tiny{{\texttt{gen}}}}}})$,   the previous  inequality can be rewritten as ${\boldsymbol\Lambda}(\{x\}^{{\mbox{\tiny{{\texttt{gen}}}}}})$ $\leq{\boldsymbol\Lambda}(C_1\cup C_2)$.   On the other hand,  ${\boldsymbol\Lambda}$ is an embedding, i.e., it is a homeomorphism onto its image, and thus $\{x\}^{{\mbox{\tiny{{\texttt{gen}}}}}}\leq C_1\cup C_2$ in $\xcal(X)$. Hence, $x\in C_1\cup C_2$, which means $x\in C_1$ or $x\in C_2$. Therefore,
\begin{equation*}
{\boldsymbol\Lambda}(\{x\}^{{\mbox{\tiny{{\texttt{gen}}}}}})\leq\sup\{{\boldsymbol\Lambda}(C_1),{\boldsymbol\Lambda}(C_2)\}.
\end{equation*}
By hypothesis, we have
\begin{equation*}
{\boldsymbol\Lambda}(C_1\cup C_2) = \sup\{{\boldsymbol\Lambda}(\{x\}^{{\mbox{\tiny{{\texttt{gen}}}}}})\mid x\in X \mbox{ such that } \lambda(x)\leq {\boldsymbol\Lambda}(C_1\cup C_2) \}\,.
\end{equation*}
Therefore, by the previous inequality, we deduce that  ${\boldsymbol\Lambda}(C_1\cup C_2)\leq\sup\{{\boldsymbol\Lambda}(C_1),$ ${\boldsymbol\Lambda}(C_2)\}$.
As observed above, the  opposite inequality also  holds, thus we have the equality, and so ${\boldsymbol\Lambda}$ is sup-preserving.

By  Theorem \ref{sigma}(3), we conclude that ${\boldsymbol\Lambda}={\boldsymbol\lambda^\sharp}$.
\end{proof}

 \begin{oss}
In general, it is possible for a spectral map $\lambda:X\rightarrow Z$ to have more than one extension $\boldsymbol{\Lambda}:\xcal(X)\rightarrow Z$, even under the hypothesis $z=\sup_Z \{\lambda(x) \mid x \in \lambda^{-1}(\{z\}^{\mbox{\tiny{{{\texttt{gen}}}}}} )\}$  (the previous proposition merely guarantees the unicity of an extension $\boldsymbol{\Lambda}$ which is an
 {\sl embedding}).

For example, suppose $Z=\xcal(X)$, and let $\lambda=\varphi$ be the canonical inclusion of $X$ in $\xcal(X)$. Clearly, if $z\in Z= \xcal(X)$, then $A:=\lambda^{-1}(\{z\}^{\mbox{\tiny{{{\texttt{gen}}}}}})$ is composed by the elements of $X$ that belong to   $\{z\}^{\mbox{\tiny{{{\texttt{gen}}}}}}$, and thus the supremum of the set $\{\lambda(x)\mid x\in A\}$, is exactly $z$. Moreover, it is clear that the homeomorphism $\boldsymbol{\lambda^{\sharp}}:\xcal(X)\longrightarrow Z = \xcal(X)$ whose existence is guaranteed by Theorem \ref{sigma} is just the identity ${\rm id}_{\boldsymbol{\mathcal{X}}(X)}$.

On the other hand, suppose that $X=\{a,b,c\}$ is composed by three elements, and endowed with the discrete topology; that is, suppose that every subset of $X$ is open. Then, $X$ is a spectral space; denote by $\boldsymbol{\Lambda}:\xcal(X)\longrightarrow\xcal(X)$ the function defined by
\begin{equation*}
\boldsymbol{\Lambda}(C):=\begin{cases}
C & \mbox{ if } C\neq \{a,b\},\\
X & \mbox{ if } C=\{a,b\}.
\end{cases}
\end{equation*}
Then, $\boldsymbol{\Lambda}$ is order-preserving (in the order induced by the Zariski topology), and since $\xcal(X)$ is finite this implies that $\boldsymbol{\Lambda}$ is continuous and spectral. Moreover, if $C$ is in $\varphi(X)$ (i.e., if $C$ is a singleton) then $\boldsymbol{\Lambda}(C)=C$: therefore, $\boldsymbol{\Lambda}\circ\varphi=\varphi={\rm id}_{\boldsymbol{\mathcal{X}}(X)}\circ\varphi$.
\end{oss}  


\section{Applications}\label{sect:appl}

In this section, we apply the topological results of the previous sections to   various algebraic settings. In particular, in Section \ref{sect:modover} we show how the construction $\xcal$ relates   a spectral space associated to a family of modules with the space of all possible intersections of the family   and we prove that the space of all overrings of an integral domain $D$ that are integrally closed
is a spectral space and it is a topological quotient of the spectral space obtained using the construction $\xcal$ from the Riemann-Zariski space $\Zar(D)$.
In Section \ref{sect:semistar} we use $\xcal$ to represent   some distinguished spaces of semistar operations, and provide a different   general proof of some results shown in \cite{fi-fo-sp-JPAA}.

\subsection{Spaces of modules and overrings}\label{sect:modover}
Let $R$ be a ring, let $M$ be an $R$-module, and let $\smod_R(M)$ be the set of $R$-submodules of $M$. The \emph{Zariski topology} on $\smod_R(M)$ is the topology having, as a subbasis of open sets, the sets in the form
\begin{equation*}
\mathtt{B}_f:=\{N\in\smod_R(M)\mid f\in N\},
\end{equation*}
where $f$ runs in  $M$;  equivalently, the sets in the form
\begin{equation*}
\mathtt{B}_F:=\{N\in\smod_R(M)\mid F\subseteq N\},
\end{equation*}
where $F$ runs among the finite subsets of $M$. Under this topology, $\smod_R(M)$ is a spectral space  \cite[Proposition 2.1]{fi-fo-sp-MANUSCRIPTA}, and the order induced by the topology is exactly the inverse of the containment order; in particular, the supremum of a subset $ \mathscr{X}\subseteq\smod_R(M)$ is exactly the intersection of the elements of $ \mathscr{X}$. Therefore, Lemma \ref{sup-spectral} translates immediately to the following.

\begin{prop}\label{prop:intersez-submod}
Let $\mathscr{X} \subseteq\smod_R(M)$ be a subset that is closed in the constructible topology  (in particular, $  \mathscr{X}$ is a spectral space). Then, the map
\begin{equation*}
\begin{aligned}
\Sigma\colon\xcal(\mathscr{X}) & \longrightarrow\smod_R(M) \\
\Delta & \longmapsto \bigcap \{N \mid N\in\Delta \}
\end{aligned}
\end{equation*}
is well-defined,   continuous, spectral, and open onto its image.  \hfill $\Box$
\end{prop}

Let now $D$ be an integral domain. An \emph{overring} of $D$ is an integral domain contained between $D$ and its quotient field $K$; the collection of all overrings of $D$ is denoted by $\Over(D)$. Under the Zariski topology, this space is closed in the constructible topology of $\smod_D(K)$ (this essentially follows from \cite{olb2016}); in particular, $\Over(D)$ is a spectral space and a subbase for the open sets of $\Over(D)$ is formed by the sets in the form
\begin{equation*}
\mathtt{O}_F:=\{B\in \Over(D)\mid B\supseteq F \},
\end{equation*}
where $F$ runs among the finite subsets of $K$. 

A distinguished subset of $\Over(D)$ is the \emph{Riemann-Zariski space} of $D$, i.e., the space $\zar(D)$ of all the valuation overrings of $D$. Then, $\Zar(D)$ is a closed set in the costructible topology of $\Over(D)$ (and thus of $\smod_D(K)$), and in particular it is a spectral space.

Part (2) of the following proposition can also be proved directly, with the same methods used to show that $\Over(D)$ is a spectral space; see \cite[Proposition 3.5 and 3.6]{Fi}.
\begin{prop}\label{prop:XZar}
Let $D$ be an integral domain, and let $ \mathscr{X}:= \overric(D)$ $ \subseteq\Over(D)$ be the space of overrings of $D$ that are integrally closed.
\begin{enumerate}
\item[\emph{(1)}] $ \mathscr{X}$ is a topological quotient of $\xcal(\Zar(D))$.
\item[\emph{(2)}] $ \mathscr{X}$ is closed in the constructible topology of $\Over(D)$; in particular, it is a spectral space.
\item[\emph{(3)}] If $D$ is a Pr\"ufer domain, then $\Over(D)$ is homeomorphic to $\xcal(\Zar(D))$.
\end{enumerate}
\end{prop}
\begin{proof}
(1)  Consider the map
\begin{equation*}
\begin{aligned}
\lambda\colon\xcal(\Zar(D)) & \longrightarrow\smod_D(K) \\
Y & \longmapsto \bigcap \{V \mid V\in Y\}.
\end{aligned}
\end{equation*}
By Proposition \ref{prop:intersez-submod}, $\lambda$ is well-defined, continuous, spectral, and open onto its image. n Moreover, the image of $\lambda$ is exactly $ \mathscr{X}$: indeed, any intersection of valuation domains is an integrally closed ring, while if $T\in  \mathscr{X}$ then $T=\lambda(\Zar(T))$, and $\Zar(T)$ is an inverse-closed subset of $\overr(D)$ (being quasi-compact and closed   under generizations). Therefore, $ \mathscr{X}$ is a topological quotient of $\xcal(\Zar(D))$,   since  the map $\lambda: \xcal(\Zar(D)) \rightarrow  \mathscr{X}$ is open, continuous and surjective.

 (2) $\mathscr{X}$  is closed in the constructible topology of $\smod_D(K)$ and of $\Over(D)$,   since it is the image of the spectral map $\lambda$.

(3) Assume that $D$ is a Pr\"ufer domain.
 We claim that $\lambda$ establishes a homeomorphism between $\xcal(\Zar(D))$ and $\Over(D)$. Indeed, since $D$ is Pr\"ufer, every overring of $D$ is integrally closed \cite[Theorem 26.2]{gi}, and thus the image of $\lambda$ is exactly $\Over(D)$.

Let now $C_1,C_2\in \xcal(\Zar(D))$ be such that $R:=\bigcap\{V\mid V\in C_1 \}=\bigcap\{V\mid V\in C_2 \}$. Since $R$ is itself a Pr\"ufer domain (as an overring of a Pr\"ufer domain), it is vacant and thus, by \cite[Corollary 4.16]{fifolo2}, $C_1,C_2$ are dense subspaces of $\zar(R)$, with respect to the inverse topology of $\zar(R)$. 
Keeping in mind that $C_1,C_2\in \xcal(\Zar(D))$, it follows immediately $C_1=C_2=\zar(R)$. This proves that $\lambda$ is injective.
Therefore,  in the present situation, $ \lambda:  \xcal(\Zar(D))\rightarrow\Over(D)$ is bijective, continuous and open, and thus it is a homeomorphism.
\end{proof}

%


\subsection{Spaces of semistar operations} \label{sect:semistar}
Let $D$ be an integral domain and let $K$ be the quotient field of $D$; let $\inssubmod(D)$ be the set of $D$-submodules of $K$. A \emph{semistar operation} on $D$ is a map $\star:\inssubmod(D)\longrightarrow\inssubmod(D)$, $I\mapsto I^\star$, such that, for every $I,J\in\inssubmod(D)$, we have: ($\star_1$) $I\subseteq I^\star$; ($\star_2$) if $I\subseteq J$, then $I^\star\subseteq J^\star$; ($\star_3$) $(I^\star)^\star=I^\star$; ($\star_4$) $xI^\star=(xI)^\star$ for every $x\in K$. For the basic properties of star, semistar and closure operations  we refer to \cite{an-overrings}, \cite{an}, \cite{ep-12}, \cite{ep-15}, \cite{gi}, \cite{hk-98}, \cite{hk-01}, and \cite{ok-ma}. 

In \cite{FiSp} a natural topology, called \emph{the Zariski topology}, on the space $\inssemistar(D)$ of all the semistar operations on $D$ was defined, by declaring as a subbasis of open sets the collection of all the sets fo the type
$$
\texttt{U}_F:=\{\star\in \inssemistar(D)\mid 1\in F^\star\},
$$
where $F$ runs among nonzero $D$-submodules of $K$. 

A semistar operation $\star$ is \emph{stable} if $(I\cap J)^\star=I^\star\cap J^\star$ for all nonzero  $D$-submodules $I$ and $J$ of $K$, and is \emph{of finite type} if $I^\star=\bigcup\{J^\star\mid J\subseteq I,J\text{~finitely generated over~}D\}$ for every nonzero  $D$-submodule $I$ of $K$. By \cite[Corollary 4.4]{FiSp}, a semistar operation $\star$ is simultaneously stable and of finite type if and only if there is a quasi-compact subset $Y\subseteq\spec(D)$ such that $\star=\s_Y$, where $\s_Y$ is defined as
\begin{equation*}
I^{{\footnotesize \s}_Y}:= \bigcap \{ID_P \mid P\in Y\}.
\end{equation*}
We denote by $\inssemistabft(D)$ the set of all stable semistar operations of finite type. 

It is quite simple to show that the set  $\inssemistabft(D)$, endowed with the subspace topology induced by that of $\inssemistar(D)$, has basic open sets of the type 
\begin{equation*}
\widetilde{\mbox{\texttt U}}_J:=\{\star\in \inssemistabft(D)\mid 1\in J^\star\},
\end{equation*}
as $J$ ranges among the nonzero finitely generated  ideal of $D$ (see \cite[Proposition 4.1(1)]{fi-fo-sp-JPAA}). Under this topology, $\inssemistabft(D)$ is a spectral space \cite[Theorem 4.6]{fi-fo-sp-JPAA} that can be thought of as a natural ``extension'' of $\spec(D)$, since the canonical map $\s\colon\spec(D)  \rightarrow\inssemistabft(D) $, defined by 
$P  \mapsto \s_{\{P\}}$,
is a topological embedding. The construction $\xcal$ introduces a new way to represent $\inssemistabft(D)$.

\begin{prop}\label{prop:invX<->semistabtf}
Let $D$ be an integral domain. 
\begin{enumerate}
\item[\rm{(1)}]
The map
$\boldsymbol{\s^\sharp}\colon\xcal(D)  \rightarrow \inssemistabft(D)$, defined by
$Y   \mapsto \s_Y$, 
and the map
$\Delta \colon$ $ \inssemistabft(D)   \rightarrow\xcal(D) $, defined by
$\star \mapsto \qspec^\star(D) :=\{ P\in \spec(D) \mid P^\star \cap D$ $=P \}$,
are homeomorphisms and  are inverses of each other. 
\item[\rm{(2)}]
If $\varphi:\spec(D) \rightarrow \xcal(D)$ is defined by  $P \mapsto \{P\}^{\mbox{\tiny\emph{\texttt{gen}}}}$ and $\s: \spec(D)\rightarrow \inssemistabft(D)$ is defined by $P \mapsto \s_{\{P\}}$, then $\boldsymbol{\s^\sharp} \circ \varphi = \s$.
\end{enumerate}
\end{prop}
\begin{proof}
(1) The fact that $\boldsymbol{\s^\sharp}$ and $\Delta$ are well-defined and bijective follows from \cite[Corollary 4.4, Proposition 5.1 and Corollary 5.2]{FiSp}.

Let $\widetilde{\mbox{\texttt{U}}}_J$ be a subbasic open set of $\inssemistabft(D)$, where $J$ is a nonzero finitely generated ideal of $D$.
 Then, $\boldsymbol{\s^\sharp}(Y)\in \widetilde{\mbox{\texttt{U}}}_J$ if and only if $1\in J^{\tiny{\s}_{\!_Y}}$, that is, if and only if $Y\subseteq \texttt{D}(J)$. Thus, by definition, $\boldsymbol{\s^\sharp}^{-1}(\widetilde{\mbox{\texttt{U}}}_J)=\boldsymbol{\mathcal{U}}(\texttt{D}(J))$ is open.

Conversely, a subbasic open set of $\boldsymbol{\mathcal{X}}(D)$ has the form $\boldsymbol{\mathcal{U}}(\texttt{D}(J))$ for some nonzero finitely generated ideal $J$.	
 As above, $\boldsymbol{\s^\sharp}(\boldsymbol{\mathcal{U}}(\texttt{D}(J)))=\boldsymbol{\s^\sharp}(\boldsymbol{\s^\sharp}^{-1}(\widetilde{\mbox{\texttt{U}}}_J))= \widetilde{\mbox{\texttt{U}}}_J$, so that $\boldsymbol{\s^\sharp}$ is open. 
 Hence, $\boldsymbol{\s^\sharp}$ is a homeomorphism.
The bijective map $\Delta: \inssemistabft(D) \rightarrow \boldsymbol{\mathcal{X}}(D)$  is also a homeomorphism, since it is the inverse map of $\boldsymbol{\s^\sharp}$ which is, in particular, continuous and open.

(2)  follows from the fact that $\s_{\{P\}}$ and $\s_{\{P\}^{\mbox{\tiny{\texttt{gen}}}}}$ coincide,  for each prime ideal $P$ of $D$.
\end{proof}

\begin{cor}
Let $D$ be an integral domain. Then, $\inssemistabft(D)$ is a spectral space.
\end{cor}
\begin{proof}
Immediate from Propositions \ref{prop:invX<->semistabtf} and \ref{embedding}.
\end{proof}

\begin{cor}\label{cor:xcal->specOmef}
Let $D_1,D_2$ be two integral domains. Then, $\spec(D_1)$ and $\spec(D_2)$ are homeomorphic if and only if so are $\inssemistabft(D_1)$ and $\inssemistabft(D_2)$.
\end{cor}
\begin{proof}
By Proposition \ref{prop:invX<->semistabtf}, $\inssemistabft(D_i)\simeq\xcal(\spec(D_i))$ for $i=1,2$. The claim now follows from Proposition \ref{psi}(3).
\end{proof}
Using a classical terminology which goes back to W. Krull (and later adjusted by R. Gilmer), a semistar operation $\star$ on $D$ is called   {\it endlich arithmetisch brauchbar}, for short \texttt{e.a.b.},   if, given finitely generated nonzero $D$-submodules $F,G,   H$ of $K$, then $(FG)^\star\subseteq (FH)^\star$ implies $G^\star\subseteq H^\star$.
Note that any $Y\subseteq \zar(D)$ induces a semistar operation $\wedge_Y$ on $D$ defined by $E^{\wedge_Y}:=\bigcap \{EV\mid V \in Y\}$, for each $E\in \FF(D)$, and a semistar operation of type  $\wedge_Y$ is  \texttt{e.a.b.} \cite[Proposition 7]{folo}. 
We denote by $\insfineab(D)$ the set of all the semistar operations that are at the same time \texttt{e.a.b} and of finite type, endowed with the subspace Zariski topology induced by that of $\inssemistar(D)$.

\begin{teor}\label{fin-eab}
Let $D$ be an integral domain. Then, the map 
$$\epsilon: \xcal(\zar(D))\rightarrow \emph{\insfineab}(D)\,,$$
defined by  $\epsilon(Y):=\wedge_Y$, for each $Y \in \xcal(\zar(D))$,  is a homeomorphism. 
 \end{teor}
\begin{proof}
Let $K$ be the quotient field of $D$, let $V$ be a valuation overring of $D$, and let $\mathfrak m_V$ be the maximal ideal of $V$. Then, the localization $V(T):=V[T]_{\mathfrak m_V[T]}$ of the polynomial ring $V[T]$ is a valuation domain of $K(T)$, called \emph{the trivial extension of $V$ to $K(T)$} \cite[Proposition 18.7]{gi}.
 For any nonempty subspace $Y$ of $\zar(D)$, consider the following subring of $K(T)$
$$
\kr(Y):=\bigcap \{V(T) \mid V \in Y\},
$$
\cite{hk-03}, \cite{hk-15}.
 In particular, set $R:=\kr(\Zar(D))$. Then, $R$ (like the $\kr(Y)$) is a B\'ezout domain with quotient field $K(T)$ \cite[Theorems 5.1 and 3.11(3)]{folo-2001}, such that $\Zar(R)$ consists of the   trivial  extensions of the valuation domains in $\zar(D)$ \cite[Propositions 3.2(2,5), 3.3 and Corollary 3.6(2)]{fifolo2}; in particular, $\Zar(R)$ is homeomorphic to $\Zar(D)$, and thus (by Proposition \ref{psi}(1)) $\xcal(\Zar(R))\simeq\xcal(\Zar(D))$. By Proposition \ref{prop:XZar}, the map  $\lambda:\xcal(\Zar(R))\longrightarrow\overr(R)$, defined by $\lambda(Z):= \bigcap\{V\mid V\in Z\}$, for each $Z \in \xcal(\Zar(D))$,  is a homeomorphism. Therefore, every overring of $R$ is in the form $\Kr(Y)$  for a unique closed set  $Y\subseteq\Zar(D)$, with respect to the inverse topology, and thus the claim will follow if we prove that the map
$
\epsilon_0\colon\overr(R)   \rightarrow \insfineab(D)$, defined by setting
$\epsilon_0(\Kr(Y)):= \wedge_Y
$, for each $ Y\in\xcal(\zar(D))$, 
is a homeomorphism.

 By \cite[Corollary 4.17]{fifolo2},  $\epsilon_0$ is clearly well-defined; it is also injective by \cite[Remark 3.5(b)]{folo-2001}. If now $\star\in\insfineab(D)$, there must be a  quasi-compact  subspace $Y$ of $\Zar(R)$ such that $\star=\wedge_Y$ \cite[Theorem 4.13]{fifolo2}, and thus $\star=\epsilon_0(\Kr(Y^{\mbox{\tiny\texttt{gen}}}))$. Hence, $\epsilon_0$ is bijective.

To show that $\epsilon_0$ is continuous, take a nonzero finitely-generated fractional ideal $F=(f_0,f_1,\ldots,f_n)D$ of $D$, and let $\texttt{V}_F=\texttt{U}_F\cap\insfineab(D)$. By \cite[Corollary 3.4(3) and Theorem 3.11(2)]{folo-2001}, we have $F^{\wedge_Y}=f\kr(Y)\cap K$, where $f$ is the polynomial $f_0+f_1T+\cdots+f_nT^n$; therefore,
\begin{equation*}
\begin{array}{rl}
\epsilon_0^{-1}(\texttt{V}_F)=& \hskip -4pt \{A\in\overr(R)\mid 1\in FA\cap K\} = \\
=& \hskip -4pt  \{A\in\overr(R)\mid f^{-1}\in\kr(Y)\}=\texttt{O}_{f^{-1}},
\end{array}
\end{equation*}
which is, by definition, an open set of $\overr(R)$.

Let now $\texttt{O}_G$ be a subbasic open set of $\overr(R)$, where $G$ is a  nonzero  finite subset of $K(T)$. Since $R$ is a B\'ezout domain, $GR=\gamma R$ for some $\gamma   \in K(T)$; therefore, $\texttt{O}_G=\texttt{O}_\gamma$.   Let  $\alpha,\beta\in K[T]$ be  two nonzero polynomials such that $\gamma=\alpha/\beta$; then,
\begin{equation*}
\begin{array}{rl}
\epsilon_0(\texttt{O}_{\alpha/\beta})=& \hskip -4pt \{\wedge_Y\in\insfineab(R)\mid \alpha/\beta\in\kr(Y)\} =\\
=& \hskip -4pt \{\wedge_Y\in\insfineab(R)\mid \alpha\in\beta\kr(Y)\}.
\end{array}
\end{equation*}
With the same reasoning as above, there are $b_0,b_1,\ldots,b_m, \ a_0,a_1,\ldots,a_k$ $\in K$ such that 
$$
\beta\kr(Y)=(b_0,b_1,\ldots,b_m)\kr(Y) \; \mbox{   and  } \; \alpha\kr(Y)=(a_0, a_1,\ldots,a_k)\kr(Y);
$$
 therefore,
\begin{equation*}
\begin{array}{rcl}
\epsilon_0(\texttt{O}_{\alpha/\beta}) \hskip -8 pt &  = & \hskip -7 pt \{\wedge_Y\!\in\insfineab(R)\!\mid\! a_0,a_1,\ldots,a_n\in\beta\kr(Y)\}=\\
& = &\hskip -7 pt \{\wedge_Y\!\in\insfineab(R)\!\mid\! a_0,a_1,\ldots,a_n\in(b_0,\ldots,b_m)\kr(Y)\!\cap\! K\}\!=\\
& = &\hskip -7 pt\{\wedge_Y\!\in\insfineab(R)\!\mid\! a_0,a_1,\ldots,a_n\in(b_0,b_1,\ldots,b_m)^{\wedge_Y}\}=\\
& = &\hskip -7 pt \bigcap_{i=0}^n\{\wedge_Y\!\in\insfineab(R)\!\mid\! a_i\in(b_0,b_1,\ldots,b_m)^{\wedge_Y}\}.
\end{array}
\end{equation*}
The sets in the last line are each equal to $\texttt{V}_{a_i^{-1}(b_0\ldots,b_m)}$; in particular, $\epsilon_0(\texttt{O}_{\alpha/\beta})$ is an intersection of a finite number of open sets, and thus open. It follows that $\epsilon_0$ is open and thus a homeomorphism, as claimed.
\end{proof}




\end{document}